\documentclass[11pt]{amsart}
\usepackage{amsmath, amssymb, amscd, mathrsfs, slashed, enumerate, url}
\usepackage{xcolor}
\usepackage[all,cmtip]{xy}
\usepackage{multicol}
\usepackage{indentfirst}
\usepackage{latexsym}
\usepackage{bm}
\usepackage{graphicx}
\usepackage{subfigure,esint}
\usepackage{float}
\usepackage{verbatim}
\usepackage{comment}
\usepackage[top=1in, bottom=1in, left=1.25in, right=1.25in]{geometry}
\usepackage{epsfig,dsfont,amsthm,amsfonts,amsbsy}
\usepackage[colorlinks]{hyperref}
\newcommand{\gpp}{\mathfrak{g}_P}
\newcommand{\gee}{\mathfrak{g}_E}

\newcommand{\MGC}{\mathcal{G}_{\mathbb{C}}}

\newcommand{\lan}{\langle }
\newcommand{\ran}{\rangle}

\newcommand{\ML}{\mathcal{L}}
\newcommand{\MD}{\mathcal{D}}

\newcommand{\MM}{\mathcal{M}}

\newcommand{\MO}{\mathcal{O}}

\newcommand{\Ker}{\mathrm{Ker}}

\newcommand{\End}{\mathrm{End}}

\newcommand{\MW}{\mathcal{W}}
\newtheorem{theorem}{Theorem}[section]

\newtheorem{conjecture}[theorem]{Conjecture}
\newtheorem{corollary}[theorem]{Corollary}

\newtheorem{definition}[theorem]{Definition}

\newtheorem{lemma}[theorem]{Lemma}

\newtheorem{proposition}[theorem]{Proposition}
\newtheorem*{remark}{Remark}

\newcommand{\MC}{\mathcal{C}}
\newcommand{\Tr}{\mathrm{Tr}}
\newcommand{\pab}{\bar{\partial}}
\newcommand{\st}{\star}
\newcommand{\we}{\wedge}
\newcommand{\pa}{\partial}

\newcommand{\RP}{\mathbb R^+}% {(0,+\infty)}
\newcommand{\EBE}{extended Bogomolny equations\;}
\newcommand{\ti}{\times}

\newcommand{\Si}{\Sigma}
\newcommand{\bz}{\bar{z}}
\newcommand{\vp}{\varphi}
\newcommand{\MA}{\mathcal{A}}
\newcommand{\ME}{\mathcal{E}}
\newcommand{\da}{\dagger}
\newcommand{\al}{\alpha}

\newcommand{\na}{\nabla}
\newcommand{\ep}{\epsilon}
\newcommand{\hA}{\widehat{A}}

\newcommand{\calC}{\mathcal C}

\newcommand{\calU}{\mathcal U}
\newcommand{\calV}{\mathcal V}

\newcommand{\calX}{\mathcal X}

\newcommand{\del}{\partial}
\newcommand{\RR}{\mathbb R}
\newcommand{\CC}{\mathbb C}

\newcommand{\NP}{\mathrm{NP}}

\newcommand{\GNP}{\mathrm{GNP}}

\newcommand{\wt}{\widetilde}
\newcommand{\hD}{\hat{D}}
\newcommand{\lam}{\lambda}

\newcommand{\si}{\sigma}
\newcommand{\isu}{i\mathfrak{su}}
\newcommand{\mft}{\mathfrak{t}}
\newcommand{\mfe}{\mathfrak{e}}
\newcommand{\HM}{H_{\mathrm{mod}}}
\newcommand{\AM}{A_{\mathrm{mod}}}
\newcommand{\PM}{\phi_{z,\mathrm{mod}}}
\newcommand{\POM}{\phi_{1,\mathrm{mod}}}

\newcommand{\MCK}{\mathcal{M}^{\mathbb{C}}_{\mathrm{Knot}}}
\newcommand{\diag}{\mathrm{diag}}
\newcommand{\mo}{\mathrm{mod}}
\newcommand{\he}{\hat{e}}

\newcommand{\MG}{\mathcal{G}}
\newcommand{\slf}{\mathfrak{sl}}
\newcommand{\mfr}{\mathfrak{r}}
\newcommand{\vfr}{\vec{\mathfrak{r}}}
\newcommand{\su}{\mathfrak{su}}
\newcommand{\Lam}{\Lambda}
\newcommand{\SL}{\mathrm{SL}}
\newcommand{\MMC}{\mathcal{M}^{\mathbb{C}}}
\newcommand{\NPK}{\mathrm{NPK}}
\newcommand{\mfd}{\mathfrak{d}}

\newcommand{\Hit}{\mathrm{Hit}}
\newcommand{\ie}{\mathrm{ie}}
\newcommand{\Ad}{\operatorname{Ad}}
\newcommand{\ad}{\operatorname{ad}}

\usepackage[utf8]{inputenc}
\usepackage[english]{babel}
\usepackage{fancyhdr}

\begin{document}

\title[The extended Bogomolny equations, II]{The extended Bogomolny equations with generalized Nahm pole boundary conditions, II}
\author{Siqi He} 
\address{Department of Mathematics, California Institute of Technology\\Pasadena, CA, 91106}
\email{she@caltech.edu}
\author{Rafe Mazzeo}
\address{Department of Mathematics, Stanford University\\Stanford,CA 94305 USA}
\email{rmazzeo@stanford.edu}

\begin{abstract}
We develop a Kobayashi-Hitchin correspondence for the \EBE, i.e., the dimensionally reduced Kapustin-Witten equations, on the
product of a compact Riemann surface $\Si$ with $\RP_y$, with generalized Nahm pole boundary conditions at
$y=0$.  The correspondence is between solutions of these equations satisfying these singular boundary conditions
and also limiting to flat connections as $y \to \infty$, and certain holomorphic
data consisting of effective triplets $(\ME, \vp, L)$ where $(\ME, \vp)$ is a stable $\mathrm{SL}(n+1,\CC)$ Higgs pair and $L \subset \ME$
is a holomorphic line bundle.  This corroborates a prediction of Gaiotto and Witten, and is an extension of our earlier
paper \cite{HeMazzeo2017} which treats only the $\mathrm{SL}(2,\RR)$ case. 
\end{abstract}

\maketitle

\begin{section}{Introduction}	
This paper generalizes our earlier work \cite{HeMazzeo2017}  and continues the study of the three-dimensional reduction of
the Kapustin-Witten (KW) equations on manifolds of the form $\Sigma \times \RR^+_y$, where $\Sigma$ is either a compact Riemann
surface or else $\Si = \CC$, with generalized Nahm pole boundary conditions at $y=0$ and where the data converges to a 
flat $\mathrm{SL}(n+1,\CC)$ connection as $y \to \infty$.  In this paper we provide a complete existence 
and regularity theory for this problem when the underlying complex gauge group is $\mathrm{SL}(n+1,\CC)$. 

We briefly recall the broader setting, and refer to \cite{KapustinWitten2006, witten2011fivebranes, gaiotto2012knot} 
as well as \cite{MazzeoWitten2013, MazzeoWitten2017} for more details. The KW equations on a four-manifold $M$, which
involve a connection $A$ on a $G$-bundle $E$ over $M$ and an $\mathrm{ad}(E)$-valued $1$-form $\phi$, take the form
\begin{equation}
F_A - \phi \we \phi + \star d_A \phi = 0, \qquad d_A \star \phi = 0.
\label{KWE}
\end{equation}
These are actually a specialization of a $1$-parameter family of equations to one particularly interesting parameter value.
Namely, if we define the complex connection $\MA = A + i\phi$ and compute its curvature $F_{\MA}$ in the usual way,
then this one-parameter family of equations can be written as
\[
e^{i\theta} F_{\MA} = \star \overline{ e^{i\theta} F_{\MA} };
\]
this can be regarded as a complex, phase-shifted form of the (anti)self-duality equations. The equations \eqref{KWE} 
correspond to the particular value $\theta = \pi/4$.  Thus the KW equations have some features of a $G_{\CC}$ 
gauge theory.

Complex flat connections are always solutions of any of this family of equations, but to find a richer class of solutions 
we specialize to this parameter value, i.e., to consider the equations \eqref{KWE}. Witten, in a series of papers and 
lectures \cite{witten2011fivebranes, Witten2014LecturesJonesPolynomial, Witten2016LecturesGaugeTheory}, following 
on the paper of Gaiotto and Witten \cite{gaiotto2012knot}, developed a far-reaching conjecture: 
the solution spaces of these equations when $M^4 = W^3 \ti \RP$, where $W \ti \{0\}$ contains a knot $K$, 
and where we impose a certain set of singular boundary conditions along $(W \ti \{0\}) \setminus K$ and separately along $K$, 
should contain information to capture the Jones polynomial of $K$ when $W = S^3$ and to define a generalization of 
the Jones polynomial in general. More specifically, the coefficients of the Jones polynomial should equal the count
of solutions to the KW equations with these singular boundary conditions (and for bundles of different degrees). 
We refer to \cite{RyosukeEnergy,He2017} for the study of moduli space of solutions with the singular boundary conditions 
and Taubes recent dramatic advance \cite{Taubescompactness} in the study of compactness properties of these moduli spaces. 

As one step toward the daunting problem of counting solutions to \eqref{KWE}, Gaiotto and Witten \cite{gaiotto2012knot} 
proposed an Atiyah-Floer type approach: fix a Heegard spliting $W=H_1\cup_{\Si} H_2$ and suppose that we stretch the metric
transversely to $H_1 \cap H_2$ so the two handlebodies are joined by a long neck $\cong \Si\ti [-L,L]$, $L \gg 1$.   If $\MM_{\Si}$ 
denotes the $G_{\CC}$ character variety of $\Si$, then the character varieties of the $H_i$ are Lagrangians $L_1, L_2 \subset \MM_{\Si}$. 
The Lagrangian intersection Floer homology of $L_1$ and $L_2$ gives an invariant of $W$, see \cite{ManolescuAbouzaid, 
FukayaDaemi2017atiyah} for recent progress. Similarly, if $W$ contains a knot, we may position it and stretch as before 
so that the portion of the knot in the long neck consists of a set of parallel straight lines $\{p_j\} \times [-L,L]$;
these intersect $\Si \ti \{0\}$ in a finite collection of points. Let $L_3$ be the moduli space of KW equations over 
$(\Si \ti \RR) \ti\RP$ with these singular boundary conditions at $\Si \ti \RR$ as before, but assuming all data is 
invariant in the $\RR$ direction.  Now, rather than counting intersections of the Lagrangians, we count holomorphic 
triangles in $\MM_{\Si}$ which span $L_1,L_2,L_3$. 
The current form of the conjecture is that this count yields the coefficients of the Jones polynomial. We refer 
to \cite{Gukov2017bps} for an explanation of this Atiyah-Floer type approach. 
There should also be symplectic knot Floer approach to define the Jones polynomial over general 3-manifold, 
in analogy to \cite{SeidelSmithLinkinvariant,ManolescuKnot}. 

%In the present situation, the singular boundary conditions play an additional
%important role. 

All of this motivates the need to describe the moduli space of solutions to the dimensionally reduced KW equations
on $\Si \ti \RP_y$ with singular boundary conditions at $y=0$, which is the topic of the present paper.  The case
where the knot is empty is already interesting, but in this dimensionally reduced setting, knots correspond to
a collection of points on $\Si$ since these `expand' to a collection of parallel lines at $y=0$ in $\Si \ti \RR \ti \RR_+$.
Thus when we refer to knot singularities in this paper, we mean simply a finite collection of points
$\{p_1, \ldots, p_k\} \subset \Si$. 

We now write these dimensionally-reduced equations explicitly. Let $G$ be a compact semisimple Lie group, and $E$ 
a complex Hermitian $G$ vector bundle over $\Si\ti\RP$. We denote by $\mathrm{ad}(E)$ the adjoint bundle of infinitesimal
automorphisms of $E$. The extended Bogomolny equations (EBE, for short) are a system of equations for a connection 
$A$ on $E$, an $\mathrm{ad}(E)$-valued $1$-form $\phi$ and an $\mathrm{ad}(E)$-valued section $\phi_1$: 
\begin{equation}
\begin{aligned}
F_A-\phi\we\phi& =\st d_A\phi_1 \\
d_A\phi+\st [\phi,\phi_1] & =0, \\ 
d_A^{\st}\phi &=0.
\end{aligned}
\label{KW2}
\end{equation}

Note that these equations do not involve either $A_y$ or $\phi_y$, the components of these fields in the vertical ($y$) direction.
While we may remove the $A_y$ component by a gauge choice, it is impossible to gauge away $\phi_y$. However, it
turns out that from the form of the singular boundary conditions at $y=0$ and the asymptotic conditions at $y=\infty$
one may deduce a posteriori that $\phi_y \equiv 0$.  On the other hand $\phi_1$ is an extra field in the theory which cannot be
removed. 

The \EBE have some important specializations. If the solution is $\Si$-invariant, then \eqref{KW2} reduces to the Nahm 
equation \cite{nahm1980simple}. If $\phi =0$, then \eqref{KW2} reduces to the Bogomolny equations \cite{Bogomol1976stability}.  
Finally, if the solution is independent of $y$ and  $\phi_1 = 0$, then \eqref{KW2} reduces to the Hitchin equations 
\cite{Hitchin1987Selfdual}. Thus the \EBE is a hybrid of these three famous equations.

For simplicity, we  assume in this paper that $G = \mathrm{SU}(n+1)$, and that the primary data is a complex vector bundle $E$ 
of degree $0$ and rank $n+1$ endowed with a Hermitian metric.  While all the results here should
go through in a relatively straightforward way for general $G$, treating that more general case would at the
least require substantial notational changes, so we do not carry this out here. 

As explained in \cite{gaiotto2012knot}, the equations \eqref{KW2} have a `Hermitian-Yang-Mills' structure.  In the
spirit of the famous papers of Donaldson \cite{donaldson1985anti, Donaldson1987Infinite}, and Uhlenbeck-Yau 
\cite{uhlenbeck1986existence}, Gaiotto and Witten predicted a `Kobayashi-Hitchin type' correspondence between 
the moduli space of solutions of these equations satisfying appropriate boundary conditions at $y=0$ and as $y \to \infty$, modulo
unitary gauge equivalence, and a certain moduli space of holomorphic data over the Riemann surface $\Si$, modulo
complex gauge equivalence. 

We describe this correspondence more carefully.  On the gauge-theoretic side, define $\MM_{\NP}$ and $\MM_{\GNP}$ 
to equal the moduli spaces of solutions to \eqref{KW2} satisfying the Nahm pole ($\NP$) and generalized Nahm pole ($\GNP$,
which stands for Nahm pole with knot singularities) boundary conditions at $y=0$ and which converge to flat $SL(n+1,\mathbb{C})$ 
connections as $y\to \infty$, up to unitary gauge equivalence.  On the complex geometric side, we consider triplets 
$(\ME,\vp,L)$ where  $(\ME,\vp)$ is a stable $\mathrm{SL}(n+1,\CC)$ Higgs pair and $L$ is a holomorphic line subbundle of $\ME$. 
We explain in Section 4.2 that any such triplet determines a divisor $\mfd(\ME,\vp,L)$ in $\Si$ which corresponds
to the location and weighting of the knot singularities. The triple is called effective if the divisor is effective. 
Define $\MMC$ to equal the set of effective triplets modulo complex gauge equivalence. We write $\MMC_\emptyset$ 
for the subset of $\MMC$ where $\mfd(\ME,\vp,L)=\emptyset$; its complement is denoted by $\MMC_{\mathrm{Knot}}$. 
We show in Section 4.1 that $\MMC_{\emptyset}$ is isomorphic to the Hitchin component $\MM_{\mathrm{\Hit}}$ in the moduli space of
$\mathrm{SL}(n+1,\mathbb{C})$ Higgs pairs. 

Gaiotto and Witten defined maps 
\begin{equation*}
I_{\NP}:\MM_{\NP}\to \MM_{\mathrm{Hit}}, \qquad I_{\NPK}:\MM_{\NPK}\to \MMC_{\mathrm{Knot}}, 
\end{equation*}
which are explained in detail in Section 4, and conjecture that $I_{\NP}$ and $I_{\NPK}$ are injective.
In this paper we verify this conjecture. 
\begin{theorem}
\label{maintheorem} The maps $I_{\NP}$ and $I_{\NPK}$ are bijective: 
\begin{itemize} 
\item[i)] For every element in the Hitchin component $\MM_{\Hit}$, there exists a solution to \eqref{KW2} with
only a Nahm pole singularity at $y=0$;
\item[ii)] For any triple $(\ME,\vp,L)$ with effective divisor  $\mfd(\ME,\vp,L)$, there exists a solution to 
\eqref{KW2} with knot singularities at $y=0$ and with locations and weight data $\mfd(\ME,\vp,L)$;
\item[iii)] If two solutions which satisfy generalized Nahm pole boundary condition have the same image under 
$I_{\NP}$ or $I_{\NPK}$, then they are unitary gauge equivalent.
\end{itemize}
\end{theorem}

The key part of this Theorem is the proof that $I_{\NP}$ and $I_{\NPK}$ are surjective. In other words, we must show that
given any holomorphic triplet as above, there exists a solution to the \EBE on $\Si\ti\RP$ satisfying the (generalized) Nahm pole 
boundary condition at $y=0$ and with specified asymptotic limit as $y\to\infty$. 

We note that this fits into a long stream of articles concerning the existence of solutions to the Hermitian-Yang-Mills 
equations over manifolds with boundary \cite{Donaldsonboundary} or with cylindrical ends, \cite{Guo1996}, 
\cite{Owens2001}, \cite{Earp2015}, \cite{WalpuskiJacob}.

Sections 2-4 below explain the framework and formulation of this problem in more detail. The existence proof is contained
in Sections 5-9.  In our formulation we search for a Hermitian metric $H$ which satisfies a quasilinear elliptic system (with
singular boundary conditions) which we write as $\Omega_H=0$.  We use the classical method of continuity.  As a preliminary
step, given any holomorphic triplets, i.e., elements of $\MM_{\Hit}$ or $\MMC_{\mathrm{Knot}}$,  we construct a Hermitian metric
$H_0$ which satisfies the Nahm boundary conditions at $y=0$ and has the desired asymptotic limit as $y \to \infty$,
and such that $\Omega_{H_0}$ vanishes to all orders as $y \to 0$ and decays exponentially as $y \to \infty$. 
We then consider perturbations $H = H_0 e^s$ of this approximate solution.  The continuity path is a family of
equations $N_t(s) = 0$, where $t=0$ corresponds to the equation we wish to solve.  We show that there is a `trivial'
solution when $t=1$.  Openness of the set of values of $t$ for which there is a solution requires a careful study of
the linearization, which in turn relies on the theory of uniformly degenerate elliptic equations \cite{Mazzeo1991}, as
developed further and specialized to the setting of the KW equations in \cite{MazzeoWitten2013, MazzeoWitten2017}.  
The argument for closedness relies on a sequence of a priori estimates.  The $\calC^0$ estimate can be handled by
the maximum principle, and the interior higher order estimates follow from known results. Our task it to prove
the uniform decay at $y=0$ and as $y \to \infty$. The former of these requires an analysis using Morrey spaces
to control the $L^\infty$ decay rate, and a more delicate argument using scale-invariant Morrey 
spaces to control higher regularity there. From these we obtain enough information to invoke the regularity theory
of \cite{MazzeoWitten2013, MazzeoWitten2017} to prove that limits of solutions are polyhomogeneous, which
then allows us to continue the deformation in the $t$ parameter.   Finally, in Section 10, we prove uniqueness, i.e.,
part $(iii)$ of Theorem \ref{maintheorem}.  To do so we study the variation of the \EBE and construct a Donaldson type 
functional for this situation, generalizng \cite{Donaldson1987Infinite}. We conclude by showing that every solution 
is a minimizer for this convex functional. 

\textbf{Conventions.}  In estimates below, $C, C', C_0$, etc., will denote constants which arise and depend only on $\Si$,
$E$ and the background Hermitian metric $H_0$, but whose values change from one line to the next. We always specify
when a constant depends on further data. 

\textbf{Acknowledgements.}  The first author wishes to thank Ciprian Manolescu, Adam Jacob, Xuemiao Chen, Thomas Walpuski 
and Xinwen Zhu for numerous helpful discussions. The second author thanks Edward Witten for many illuminating
conversations; he has been supported by the NSF grant DMS-1608223.
\end{section}

\begin{section}{The Extended Bogomolny Equations}
	\begin{subsection}{Hermitian Geometry}
Consider the space $\Si \ti \RP$, where $\Si$ is a compact Riemann surface with product metric 
$g=g_0^2|dz|^2+dy^2$. Let $E$ be complex Hermitian vector bundle of rank $n+1$ over $\Si\ti\RP$ with $\det E=0$.
Fixing a Hermitian metric $H$ on $E$ gives an $\mathrm{SU}(n+1)$ structure on this bundle, and we denote by
$\gee$ the associated adjoint bundle. Finally, $y$ denotes a fixed linear coordinate on $\RP$, and we use any local 
holomorphic coordinate chart $z=x_2+ix_3$ on $\Si$. 

As explained in the introduction, the fields in our equation are a unitary connection $A$ on $E$, and two Higgs fields
$\phi\in\Omega^1(\gee)$ and $\phi_1\in\Omega^0(\gee)$. We work in a gauge where $A^{\st}=-A$, $\phi^{\st}=\phi$, 
$\phi_1^{\st}=-i\phi_1$; here $\st$ is the conjugate transpose determined by the Hermitian metric $H$. 

Our starting point is an observation by Gaiotto and Witten in \cite{gaiotto2012knot} that the \EBE have a Hermitian-Yang-Mills 
structure.  Write 
\[
d_A=\na_2dx_2+\na_3dx_3+\na_ydy,\ \mbox{and}\ \ \phi=\phi_2dx_2+\phi_3dx_3 =\frac{1}{2}(\vp_z dz+\vp_{\bz} d\bz),
\]
where $\vp_z=\phi_2-i\phi_3$ and $\vp_{\bz} = \phi_2 + i\phi_3$.  We often write 
$\vp_z dz$ as $\vp$ since this will be our primary object; indeed $\vp$ determines $\phi$ since $\vp_{\bz} = - \vp_z^*$. 
%We also write the holomorphic part of $\phi$ as $\vp=i\phi_zdz=(i\phi_2+\phi_3 )dz.$  

Now, following \cite{witten2011fivebranes} \cite{gaiotto2012knot}, define the operators 
\begin{equation}
\begin{split}
&\MD_1=(\na_2+i\na_3)d{\bz}=(2\pa_{\bz}+A_1+iA_2)d{\bz},\\
&\MD_2=\operatorname{ad}\vp=[\vp,\cdot ]=[(\phi_2 - i \phi_3) \, dz ,\cdot ], \\
&\MD_3=\na_y-i\phi_1=\pa_y+A_y-i\phi_1, 
\end{split}
\end{equation}
and their adjoints with respect to $H$ and the pairing $(\alpha, \beta) \mapsto \int_{\Si \ti \RP} \alpha \we \st\overline{\beta}$.  
Noting that $\MD_1 = \pab_A$ defines the holomorphic structure on $E$, then its adjoint on sections or endomorphisms
valued in $0$-forms and $(0,1)$-forms, respectively, is $(\MD_1^{\dag_H})_0 = \partial_A = (\na_2 - i \na_3) dz$
and $(\MD_1^{\dag_H})_{0,1} = - \partial_A = (-\na_2 + i \na_3) dz$. For the other two operators we have
\[
\MD_2^{\da_H}=-[\phi_2+i\phi_3, \cdot]\, d\bz,\ \MD_3^{\da_H}=-\na_y-i\phi_1.
\]
The \EBE can then be written in the elegant form
\begin{equation}
\begin{split}
&[\MD_i,\MD_j]=0, \ \ i,j=1,2,3,\\
& \frac{i}{2}\Lambda \left([\MD_1, \MD_1^{\da_H}]+[\MD_2,\MD_2^{\da_H}]\right)+ [\MD_3,\MD_3^{\da_H}]=0,
\label{EBE}
\end{split}
\end{equation}
where $\Lambda:\Omega^{1,1}\to\Omega^0$ is the inner product with the K\"ahler form (normalized as $(i/2)dz\we d\bz$ 
when the metric on $\Si$ is flat). 

In the local coordinates above, and taking into account the signs on $\MD_1^{\dag_H}$, we calculate
\begin{equation}
\begin{split}
&[\MD_1,\MD_1^{\da}]=2iF_{23}\, dz\we d\bz,\\ 
&[\MD_2,\MD_2^{\da}]=-2i[\phi_2,\phi_3]\, dz\we d\bz,\\
&[\MD_3,\MD_3^{\da}]=-2i\na_y\phi_1,
\end{split}
\end{equation}
so
\begin{equation}
\frac{i}{2} \Lambda \left([\MD_1, \MD_1^{\da}]+[\MD_2,\MD_2^{\da}]\right)  + [\MD_3,\MD_3^{\da}] =F_{23}-[\phi_2,\phi_3]-\na_y\phi_1=0.
\end{equation}

As pointed out by Witten \cite{witten2011fivebranes}, the key to understanding this system is the observation that 
the first set of equations, $[\MD_i,\MD_j]=0$, enjoy a larger symmetry than the full system. Indeed, the full system
is invariant under the real gauge group $\mathcal G$ of special unitary transformations, while the first set of equations is
invariant under the complex gauge group $\MGC:=\SL (n+1,E)$ of special linear automorphisms of $E$.   The action
is as follows: if $g\in \MGC$, then 
\begin{equation}
\MD_i^{g}:=g\circ \MD_i\circ g^{-1},\; \MD_i^{\da,g}:=(g^{\st})^{-1}\circ \MD_i\circ g^{\st},
\end{equation}
where $g^{\st}$ is the conjugate transpose of $g$. This extends the action of $\MG=\{g\in\MGC|gg^{\st}=1\}.$ 
The first set of equations in \eqref{EBE} can be regarded as a complex moment map, while the final equation
is the accompanying real moment map.

%For \eqref{EBE}, the smaller system $[\MD_i,\MD_j]=0$ is invariant under the complex gauge group $\MGC$ action invariant, but the full system is invariant under the unitary gauge group. We can visualize the last equation as a real moment map equation. 

\begin{remark}
We have used the phrase Hermitian-Yang-Mills structure for the \EBE because of the following analogy. 
Let $X$ be a K\"ahler manifold, $E$ a complex bundle over $X$ with $\det E=0$. The Hermitan-Yang-Mills 
equations for connections $A$ on $E$ take the form
\begin{equation}
F_{A}^{(0,2)}=0,\;i\Lambda F_A^{(1,1)}=0,
\end{equation}
where $\Lambda$ is the inner product with the K\"ahler class. Using local coordinates $(z_1,\cdots,z_n)$, define 
the $D_i=\pa_{\bz_i}+A_{\bz_i}$. Then the equation $F_{A}^{(0,2)}=0$ is equivalent to $[D_i,D_j]=0$ for all $i,j$, 
while $i\Lambda F_A^{(1,1)}=0$ is equivalent to $\sum_{i=1}^n[D_i,D_i^{\da}]=0$.
\end{remark}
\end{subsection}

\begin{subsection}{Holomorphic Data}
	\label{holdata}
Following Donaldson-Uhlenbeck-Yau \cite{donaldson1985anti}, \cite{uhlenbeck1986existence}, \cite{Donaldson1987Infinite}, 
this leads to the expectation that one should start from holomorphic data satisfying the $\MGC$-invariant equations, and then
correct these to solve the $\MG$-invariant equation.   We now show how this works in the present circumstances.

Denote by $E_y$ the restriction of $E$ to the slice $\Si\ti\{y\}$. Now observe that $\MD_1$ is a $\bar{\pa}$ operator 
which satisfies $\MD_1^2=0$, so by the Newlander-Nirenberg theorem, it defines a holomorphic structure $\ME_y$ on $E_y$.
Next, on each slice $\ME_y$, $\MD_2$ is a $K_\Si$-valued endomorphism ($K_\Si$ is the 
canonical bundle of $\Si$), and the equation $[\MD_1,\MD_2]=0$ implies $\MD_1\vp=0$, i.e., the endormophism $\vp$ 
is holomorphic. In other words, writing $\vp_y$ for the restriction of $\vp$ to $E_y$, we obtain a family of Higgs bundle 
$(\ME_y,\vp_y)$ over $\Si\ti\{y\}$. Finally, $\MD_3$ provides a parallel transport in the $y$ direction, and 
the equations $[\MD_1,\MD_3]=0,[\MD_2,\MD_3]=0$ imply that this family of Higgs bundles are parallel with respect
to $\MD_3$. 

Based on these observations, we define a data set for our problem to consist of a bundle $E$ of rank $(n+1)$ and degree $0$ 
over $\Si\ti\RP$, along with a system of operators $\Theta=(\MD_1,\MD_2,\MD_3)$ acting on $\MC^{\infty}(\Si \ti \RP; E)$ which
satisfy:
\begin{itemize}
\item For any smooth function $f$ and section $s$ of $E$, $\MD_1(fs) = \bar{\pa}f s + f \MD_1 s$, 
$\MD_3(fs) = (\del_{y}f) s + f \MD_3 s$; 
\item $\MD_2=[\vp, \cdot ]$ for some $\vp\in \Omega^{1,0}(\gpp)$; 
\item $[\MD_i,\;\MD_j]=0$ for all $i, j$. 
\end{itemize}
Two sets of data $(E,\Theta)$ and $(E, \wt \Theta)$ are called equivalent if there exists a complex gauge transform $g$ such that 
$g^{-1}\wt \MD_i g=\MD_i$, $i=1,2,3$.  By our previous discussion, under the $\MGC$ action, $(E,\Theta)$ is equivalent 
to a Higgs bundle. We discuss Higgs bundles further in Section 3. 

Given $(E,\Theta)$, a choice of Hermitian metric $H$ on $E$ determines the adjoints $\MD_i^{\da_H}$ of the operators $\MD_i$ by
\begin{itemize}
\item $\bar{\pa}(H(s,s'))=H(\MD_1 s,s')+H(s,\MD_1^{\da_H}s'),\;\pa_{y}(H(s,s'))=H(\MD_3 s,s')+H(s,\MD_3^{\da_H}s')$;
\item $H(\MD_2 s,s')+H(s,\MD_2^{\da_H}s') = 0$.
\end{itemize}
We then have that $\MD_1^{\da_H}$ and $\MD_3^{\da_H}$ are derivations and $\MD_2^{\da_H}$ is 
tensorial, i.e., $\MD_1^{\da_H} (fs) = (\del_{z} f) s + f \MD_1^{\da_H} s$, 
$\MD_3^{\da_H} (fs) = (\del_{y} f) s + f \MD_3^{\da_H}s$, while $\MD_2^{\da_H}(fs) = f \MD_2^{\da_H}(s)$.  
%Note also that $\MD_i'=-\MD_i^{\da}$. \rafe{$\MD_i' = \MD_i^{\da_H}$, notits negative?} 

We now shift perspective slightly. Instead of letting complex gauge transformations act on the data set $\Theta$, it 
is easier to fix $\Theta$ and let $\MGC$ act on the Hermitan metric.  We thus regard the real moment map equation 
in \eqref{EBE} as an equation for the Hermitian metric $H$.  Set $\MD_y=\frac{1}{2}(\MD_3+\MD_3^{\da_H})$, 
$\MD_{\bz}=\MD_1$ and $\MD_{z}=\MD_1^{\da_H}$ and define a unitary connection $\MD_{A}$, and an
endomorphism-valued $1$-form $\phi$ and $0$-form $\phi_1$ on $(E,\Theta,H)$ by 
\begin{equation}
\begin{split}
\MD_A s:&=\MD_1s+\MD_1^{\da_H}s+\MD_y s\, dy,\\
[\phi,s]:&=[\MD_2,s]+[\MD_2^{\da_H},s],\\
\phi_1:&=\frac{i}{2}(\MD_3-\MD_3^{\da_H}).
\end{split}
\label{relationship}
\end{equation}
The triple $(A,\phi,\phi_1)$ is called the Chern connection of $(E,\Theta,H)$. 

The following proposition about the $\MGC$ action on $\Theta$ was proved in {\cite{HeMazzeo2017}}:
\begin{proposition}
\label{dataonlycomesfromequation}
(1) Suppose that $(E,\Theta)$ and $(E, \wt \Theta)$ are two data sets. %Define $E_y$ to be the slice $E_y:=E|_{\Si\ti\{y\}}$.
If the restrictions of $\Theta$ to $E_y$ and $\wt \Theta$ to some possibly different $E_{y'}$ are complex gauge equivalent, 
then $(E,\Theta)$ and $(E, \wt \Theta)$ are equivalent. 

(2) If $(E,\Theta,H)$ is a solution to the \EBE, and if $g$ is a complex gauge transform, then $(E,\Theta^g)$, where $\Theta^g=
(g^{-1}D_1g,g^{-1}D_2g,g^{-1}D_3g), H^g =Hg^{\st_H} g$ is also a solution.
\end{proposition}

We now record some computations in a local frame, with coordinate $(z,y)$. Writing $\MD_1=\pa_{\bz}+\al$, 
$\MD_1^{'}=\pa_{z}+A^{(1,0)}$, $\MD_3=\pa_y+\MA_y$ and $\MD_3^{'}=\pa_y+\MA_y^{'}$, we compute:
\begin{equation}
\begin{split}
A^{(1,0)}&=H^{-1}\partial_z H-H^{-1}(\bar{\al})^{\top}H,\\
A&=A^{(1,0)}+\alpha=H^{-1}\partial_z H-H^{-1}\bar{\al}^{\top}H+\al,\\
\vp^{\da}&=H^{-1}\bar{\vp}^{\top}H,\\
\MA_y^{'}&=H^{-1}\partial_y H-H^{-1}\bar{\MA_y}^{\top} H.
\end{split}
\label{99}
\end{equation}
Thus in a frame where $\alpha=\MA_y=0$, then %the adjoint operators become
\begin{equation}
\begin{split}
\MD_1^{\da}=-\MD_1^{'}=-(\pa_z+H^{-1}\pa_z H),\;
\MD_2^{\da}=-\MD_2^{'}=[\vp^{\da}, \cdot ],\;
\MD_3^{\da}=-\MD_3^{'}=-\pa_y-H^{-1}\partial_y H.
\end{split}
\end{equation}

The gauge is called holomorphic parallel if $\MD_1=\bar{\pa}$, $\MD_3=\pa_y$. In such a gauge, 
the moment map equation \eqref{EBE} becomes
\begin{equation}
-\bar{\pa}(H^{-1}\pa H)-g_0^2\pa_y(H^{-1}\pa_y H)+[\vp,\vp^{\star}]=0,
\end{equation}
where the metric on $\Si$ is $g_0^2 |dz|^2$. 

Next, given $(E,\Theta,H)$, consider the Chern connection $(A,\phi,\phi_1)$. The gauge is called unitary if $(A,\phi,\phi_1)$ 
are unitary matrices.  In analogy to a standard result \cite{atiyah1978geometry}, we record the link between connections 
in unitary and holomorphic frames: 
\begin{proposition}{\cite{HeMazzeo2017}} \label{complexgaugeactionchange}
With $(E,\Theta, H)$ as above, there is a unique triplet $(A,\phi,\phi_y)$ compatible with the unitary structure and 
with structure defined by $\Theta$. In other words, in every unitary gauge, $A^{\st}=-A$, $\phi^{\st}=\phi$, $\phi_1^{\st}=-\phi_1$, 
while in every parallel holomorphic gauge, $\MD_1=\overline{\partial}_E$ and $\MD_3=\pa_y$, i.e., $A^{(0,1)}= A_y-i\phi_1=0$.
\end{proposition}

In a local holomorphic trivialization of $E$, we can represent the metric by a Hermitian matrix (also denoted $H$). 
For $g\in\MGC$ with $g^{\da}g=H$, e.g.\, $g=H^{\frac{1}{2}}$, then in holomorphic parallel gauge 
\begin{equation}
A^{(1,0)}=H^{-1}\pa H =g^{-1}(g^{\da})^{-1}(\pa_z g^{\da}) g+g^{-1} \pa_z g, \ \ A^{(0,1)}=0.
\end{equation}
Thus $g$ transforms from holomophic to unitary gauge. If $\hA$ is the connection form in unitary gauge, then 
\begin{equation}
\hA_z=(g^{\da})^{-1}\pa_zg^{\da}, \ \ \hA_{\bar{z}}=-(\pa_{\bar{z}}g) g^{-1},
\end{equation}
so $\hA_{\bar{z}}^{\da}=-\hA_z$. Similarily, 
\begin{equation}
\begin{split}
&\phi_z=g\vp g^{-1},\;\phi_{\bz}= (g^{\da})^{-1}\bar{\vp}^{\top} g^{\da},\\
&A_y=\frac{1}{2}((\pa_yg) g^{-1}-(g^{\da})^{-1}\pa_yg^{\da}),\ \ \phi_1=\frac{i}{2}( (g^{\da})^{-1}\pa_yg^{\da}+\pa_yg^{\da}
(g^{\da})^{-1}),
\end{split}
\end{equation}
where $\vp_z$ is the holomorphic gauge and $\phi,\;A_y,\;\phi_1$ are in the unitary gauge. 
\end{subsection}
\end{section}

\begin{section}{Boundary Conditions}
The boundary conditions for this system are that $(A,\phi,\phi_1)$ converges to a flat irreducible $\SL (n+1,\CC)$ 
connection as $y \to \infty$, while for $y\to 0$, $(A,\phi,\phi_1)$ satisfies the generalized Nahm pole boundary condition with
knot singularities.  We describe these in more detail now.

\begin{subsection}{Higgs Bundle}
We now recall the well-known interrelationship between the moduli spaces of flat $\SL(n+1,\CC)$ connections,
Higgs bundles and solutions of the Hitchin equations. % we shall do so in the language of Higgs bundle.
Recall that a Higgs bundle (or Higgs pair) over $\Si$ is a pair $(\ME,\varphi)$, where $(\ME, \bar{\pa}_{\ME})$ 
is a holomorphic bundle of rank $n+1$ and $\vp\in H^0(\End(\ME)\otimes K)$. We assume here that $\det \ME=0$. 
A Higgs pair $(\ME,\vp)$ with $\deg \ME = 0$ is called stable if $\deg (V)<0$ for any holomorphic subbundle 
$V$ with $\varphi(V)\subset V\otimes K$, and polystable if it is a direct sum of stable Higgs pairs. 
(This is a special case of the familiar definition when $\deg \ME$ is not necessarily $0$.) 

The Hitchin equations are obtained by setting $\MD_3 = 0$ in the \EBE (or alternately, considering only the
equations for $\MD_1$ and $\MD_2$ on each slice $\Si \ti \{y\}$): 
\begin{equation}
F_H+[\vp,\vp^{\st_H}]=0,\;\bar{\pa}\vp=0.
\label{Hitchinequation}
\end{equation}
We regard the Hermitian metric $H$ on $\ME$ as the variable, and then $F_H$ is the curvature of the Chern 
connection $\nabla_H$ associated to $H$ and the holomorphic structure, and $\vp^{\st_H}$ is the 
adjoint with respect to $H$.  Irreducibility and reducibility of solutions $(A,\vp)$ is defined in the obvious way.
\begin{theorem}{\cite{Hitchin1987Selfdual}}
For any Higgs pair $(\ME,\vp)$ on $\Si$, there exists an irreducible solution $H$ to the Hitchin equations 
if and only if this pair is stable, and a reducible solution if and only if it is polystable.
\end{theorem}

To any solution $H$ to the Hitchin equations is associated a flat $\mathrm{SL}(n+1,\CC)$ connection
$D = \nabla_H + \vp + \vp^{\st_H}$, and hence a representation $\rho: \pi_1(\Si) \to \mathrm{SL}(n+1,\CC)$,
which is well-defined up to conjugation.  Irreducibility of the solution is the same as irreducibility of the representation,
while reducibility corresponds to the fact that $\rho$ is reductive.   The map from flat connections back
to solutions of the Hitchin system involves finding a {\em harmonic metric} which yields a decomposition of $D = D^{\mathrm{skew}} 
+ D^{\mathrm{Herm}}$ into skew-Hermitian and Hermitian parts, so that $D^{\mathrm{Herm}} = \vp + \vp^{\star_H}$
and $( (D^{\mathrm{skew}})^{0,1}, \vp)$ satisfies Hitchin equations.

The culmination of the work of Hitchin, Donaldson, Simpson and Corlette is the diffeomorphic equivalence between
the spaces of stable Higgs pairs, irreducible solutions of the Hitchin equations and irreducible flat connections;
there is a similar equivalence for the polystable/reducible spaces. 

The condition that $(A,\phi,\phi_1)$ converges to a flat $\mathrm{SL}(n+1,\mathbb{C})$ connection thus requires
that we only consider the stable Higgs bundle. 

Now consider the moduli space $\MM_{\mathrm{Higgs}}$ of $\SL(n+1,\mathbb{C})$ Higgs bundle.  The Hitchin fibration is the map 
% $p_2,p_3,\cdots,p_{n+1}$ be a basis for the algebra of invariant polynomials on the Lie algebra of degree $2,3,\cdots,n+1$, then we obtain the holomorphic map introduced in 
\begin{equation}
\begin{split}
&\pi:\MM_{\mathrm{Higgs}}\to \oplus_{i=1}^n H^0(\Si,\;K^{n+1})\\
&\pi(\vp)=(p_2(\vp),\;\cdots,\;p_{n+1}(\vp)),
\label{HitchinFiberationMap}
\end{split}
\end{equation}
where $\det (\lambda -\vp) = \sum \lambda^{n+1-j} (-1)^j p_{j}(\vp)$. By \cite{hitchin1987stable}, this map
is proper. 

We next introduce the so-called Hitchin component (also called the Hitchin section). Choose a spin structure $K^{\frac{1}{2}}$ and, writing $B_{i}=i(n+1-i)$,
define the Higgs bundle $(\ME, \vp)$:
\begin{equation}
\begin{aligned}
\ME: & =S^{n}(K^{-\frac{1}{2}}\oplus K^{\frac{1}{2}})=K^{-\frac{n}{2}}\oplus K^{-\frac{n}{2}+1}\oplus\cdots\oplus K^{\frac{n}{2}} \\[5mm]
\vp & =\begin{pmatrix}
    0 & \sqrt{B_1} & 0 &\cdots& 0\\
    0 & 0 & \sqrt{B_2} & \cdots& 0\\
    \vdots & \vdots &\ddots & &\vdots\\
    0 & \vdots & &\ddots &\sqrt{B_n}\\
    q_{n+1}& q_{n} & \cdots & q_2 &0
\end{pmatrix}.
\end{aligned}
\label{HitchincomponentHiggsfield}
\end{equation}
The constant $\sqrt{B_i}$ in the $(i,i+1)$ entry represents this multiple of the natural isomorphism $K^{-\frac{n}{2} + i  } \to 
K^{-\frac{n}{2} + i -1 }\otimes K$, and similarly, $H^0(\Si,K^{n+1-i}) \ni q_{n+1-i} :K^{- \frac{n}{2} + i }\to K^{ \frac{n}{2}} \otimes K$. 

The complex gauge orbit of this family of Higgs bundle,
\begin{equation}
\{(\ME:=S^{n}(K^{-\frac{1}{2}}\oplus K^{\frac{1}{2}}),\ \vp\ \mbox{as in}\ \eqref{HitchincomponentHiggsfield})\}/\MGC.
\end{equation}
is called the Hitchin component, or sometimes also the Hitchin section, and denoted $\MM_{\mathrm{Hit}}$. Note that when $n$ 
is even, only even powers of $K^{1/2}$ appear, so $\MM_{\mathrm{Hit}}$ is independent of the choice of spin structure in that case. 
\begin{theorem}{\cite{hitchin1992lie}}
(1) Every element in $\MM_{\mathrm{Hit}}$ is a stable Higgs pair, and $\MM_{\mathrm{Hit}}$ is parametrized by the vector space 
$\oplus_{i=2}^{n+1}H^0(\Si,K^{i})$;

(2) The map which assigns to an element of $\oplus_{i=2}^{n+1}H^0(\Si,K^{i})$ the unique corresponding solution of the Hitchin equations
is an equivalence from this vector space to one component of the moduli space of flat completely irrreducible 
$\SL(n+1,\mathbb{R})$ connections. In other words, the map
$$
\pi|_{\MM_{\mathrm{Hit}}}:\MM_{\mathrm{Hit}}\to \oplus_{i=2}^{n+1}H^0(\Si,K^{i})
$$ 
is a diffeomorphism. 
\end{theorem}
Part (2) here explains why we call $\MM_{\mathrm{Hit}}$ a component .
\end{subsection}

\begin{subsection}{Nahm Pole Boundary Condition}
We now introduce the Nahm pole boundary condition for bundles of rank $n+1$, cf.\ \cite{witten2011fivebranes}. 

Recall some basic representation theory for $\slf_{n+1}$, cf.\ \cite{Serre2001Complex} for terminology. 
The Cartan matrix $A$ of $\slf_{n+1}$ is the $n\times n$ matrix
\begin{equation}
A = (A_{ij})  = \begin{pmatrix}
2 & -1 & 0 & 0&\cdots& 0\\
-1 & 2 & -1 & 0&\cdots& 0\\
0 & -1& 2& -1&\cdots &0\\
\vdots & \vdots &\ddots &  &\vdots\\
0 & 0 & 0&\ddots&2  &-1\\
0& 0 & 0&\cdots  &-1 &2
\end{pmatrix},
\label{Cartanmatrix}
\end{equation} 
%$A^{-1}_{ij}$ for the $ij$ entries of $A$ and $A^{-1}$. 
and there is an explicit formula for the components of its inverse
\begin{equation}
A_{ij}^{-1}=\min\{ij\}-\frac{ij}{n+1},
\end{equation}
cf.\ \cite{Wei2017Inverse}. The constants
\[
B_i=2\sum_j (A^{-1})_{ij}=i(n+1-i)
\] 
appear in several places below.

We use a Chevalley basis in the Lie algebra, with standard representatives
%with $E_i^{\pm}$ the raising and lowering operators for each simple root
%and $H_i$ the corresponding coroot. 
\[
E^+_j :=E_{j,j+1}, \ \ E^-_j:=E_{j+1,j}, \ \  H_j:=E_{j,j}-E_{j+1,j+1}
\]
where $E_{ij}$ is the matrix with a $1$ in the $ij$ slot and $0$'s elsewhere. These have commutation relations
\begin{equation}
[H_i,H_j]=0,\;[H_i.E_j^{\pm}]=\pm A_{ij}E_j^{\pm},\;[E_i^+,E_j^-]=\delta_{ij}H_j.
\label{ChevalleyBasis}
\end{equation}

Fix a principal embedding of $\mathfrak{sl}_2$ in $\mathfrak{sl}_{n+1}$ and let $\mathfrak{t}_1,\mathfrak{t}_2,\mathfrak{t}_3$ be 
a basis of the embedded $\mathfrak{sl}_2$ satisfying $[\mathfrak{t}_i,\mathfrak{t}_j]=\ep_{ijk}\mathfrak{t}_k$. Let $E$ be
a degree $0$ bundle of rank $n+1$ over $\mathbb{C}\ti\RP$. The \EBE \eqref{EBE} have the singular model solution 
\begin{equation}
A=0,\;\phi_z=\frac{\mathfrak{t}_1-i\mathfrak{t}_2}{y},\;\phi_1:=\frac{\mathfrak{t}_3}{y},
\label{NahmSolution}
\end{equation}
which are in fact the basic solution of the Nahm equations \cite{nahm1980simple}.  There are good representatives of this
conjugacy class: 
\begin{equation}
\mathfrak{t}_1-i\mathfrak{t}_2=\sum_j \sqrt{B_j}E_j^+, \ \mathfrak{t}_3=\frac{i}{2}\sum_j B_jH_j.
\end{equation}
We also set $\mfe^+:=\mft_1-i\mft_2,\;\mfe^-:=-\mft_1-i\mft_2,\mfe^0:=-2i\mft_3$, so that
\begin{equation}
[\mfe^+,\;\mfe^-]=\mfe^0,\;[\mfe^0,\;\mfe^-]=-2\mfe^-,\;[\mfe^0,\;\mfe^+]=2\mfe^+.
\end{equation}

Now define 
\begin{equation}
\begin{split}
\phi_z= \frac{\mfe^+}{y} \, dz  = \frac{1}{y}\begin{pmatrix}
0 & \sqrt{B_1} & 0  &\cdots& 0\\
0 & 0 & \sqrt{B_2}  &\cdots& 0\\
\vdots & \vdots &\ddots &  &\vdots\\
\vdots & \vdots & &\ddots  &\sqrt{B_n}\\
0& 0 & \cdots & 0 &0
\end{pmatrix}\, dz,
\end{split}
\label{modelformphiz}
\end{equation}
and
\begin{equation}
\begin{split}
\phi_1= \frac{i \mfe^0}{2y} = \frac{i}{2y}\begin{pmatrix}
n & 0 & 0  &\cdots& 0\\
0 & n-2 &  0 &\cdots& 0\\
\vdots & \vdots &\ddots &  &\vdots\\
\vdots & \vdots & 0&-(n-2)  &0\\
0& 0 & \cdots  &0 &-n
\end{pmatrix}.
\end{split}
\label{cannonicalformsl2Liealgebra}
\end{equation}

%Over $\Si\ti\RP$, where $\Si$ is a Riemann surface with genus $g(\Si)>1$, in \cite{witten2011fivebranes}, Witten introduce a singular boundary condition with leading asymptotic as the solution in \eqref{NahmSolution}:
\begin{definition}
We say the fields $(A,\phi_z,\phi_1)$ on the bundle $E$ %with $\det E=0$ over $\Si\ti\RP$ 
satisfies the Nahm pole boundary condition if in terms of some local trivialization  
\begin{equation}
A\sim \MO(y^{-1+\ep}),\;\phi_z=\frac{\mfe^+}{y}+\MO(y^{-1+\ep}),\;\phi_1:=\frac{i\mfe^0}{2y}+\MO(y^{-1+\ep})
\label{Nahmpoleboundaryconditionforconnection}
\end{equation}
as $y\to 0.$
\end{definition}
We may assume that the error blows up slightly less rapidly than $y^{-1}$ either in an $L^2$ or $L^\infty$ sense, but then,
as described in \cite{MazzeoWitten2013}, one may prove that if these fields satisfy the \EBE or Kapustin-Witten equations,
then there exists a gauge in which they are much more regular. 
%it is of course necessary to suppose that these fields lie in some specific function space, e.g. a weighted Holder space, and the error estimate $\MO(y^{-1+\ep})$ is interpreted in terms of that norm. The regularity theory in that paper shows that a solution of the extended Bogomolny equations, or indeed of the full Kapustin-Witten system, is then much more regular after being put into gauge.

We also study these boundary condition in parallel holomorphic gauge. Let $(\calU, z)$ be a local holomorphic chart 
on $\Si$, and as suggested in \cite{witten2011fivebranes}, consider a local holomorphic trivialization of $\ME$ in terms 
of which %the Higgs field takes the form
\begin{equation}
\label{vpformNahmpole}
\begin{split}
\vp =\begin{pmatrix}
\st & \sqrt{B_1} & 0 &\cdots& 0\\
\st & \st & \sqrt{B_2} &\cdots& 0\\
\vdots & \vdots &\ddots & &\vdots\\
\vdots & \vdots & &\ddots &\sqrt{B_n}\\
\st& \st & \cdots  &\st &\st
\end{pmatrix}\, dz,
\end{split}
\end{equation}
so all entries above the main diagonal are constant or zero.   This gives a commuting triplet $(\MD_1, \MD_2, \MD_3 = \del_y)$. 
The Nahm pole boundary condition is attained by adjoining the {\em singular} Hermitian metric $H_0= \exp(-\log y\, \mfe^0)$.
Indeed, following Proposition \ref{complexgaugeactionchange},  changing to a unitary frame for this metric corresponds to
conjugating by the complex gauge transformation $g_0$ for which $g_0^2 = H_0$. These conjugated fields have
the form \eqref{modelformphiz} (plus a term which is $\MO(1)$) and \eqref{cannonicalformsl2Liealgebra}, and hence
satisfy the Nahm pole  boundary conditions.  Consider any other Hermitian metric $H=H_0e^s$, where $s$ is a section of 
$\isu(E,H_0)$ with $\sup |s| + y |ds| \leq Cy^\ep$. A straightforward computation shows that the corresponding fields 
$(A_H,\phi_H,(\phi_1)_H)$ in unitary gauge then also satisfy the Nahm pole boundary conditions. This leads
to the definition of Nahm pole boundary condition for a Hermitian metric: 
\begin{definition}
Suppose that in some local trivialization, the Higgs field $\vp$ has the form $\eqref{vpformNahmpole}$. In that frame,
set $H_0=\exp(-\log y\, \mfe^0)$. Then we say that a Hermitian metric $H$ satisfies the Nahm pole boundary condition
if $H = H_0 e^s$ with $|s| + |y\, ds| < Cy^\ep$. 
\end{definition}
%Since the Chern connection involves one derivative of $H$, then assuming some regularity of $s$, it is not hard to
%see that these two definitions coincide.
\end{subsection}

\begin{subsection}{Knot Singularities}
	\label{subsectionknotsingularity}
We now define the singular model solution for a knot. This model was found by Witten for rank $2$ bundles
\cite{witten2011fivebranes}, while for $\mathrm{SL}(n+1)$, $n>1$, and general higher rank semisimple groups,
it was obtained by Mikhaylov \cite{Mikhaylov2012solutions}. 
The form is rather complicated and less explicit in general, hence we shall simply outline the initial reduction of the equations
and describe its relevant features here.

Let $E$ be a complex bundle of rank $n+1$ over $\mathbb{C}\ti\RP$, and consider the extended Bogomolny equations with 
singularity at $\{z=y=0\}$. This is called the boundary t'Hooft operator in the physics literature. 

Fix an $n$-tuple of nonnegative integers $\vfr = \{\mfr_1, \ldots, \mfr_n\}$ and define, in parallel holomorphic gauge and
using a Chevalley basis, the Higgs field
\begin{equation}
\label{varphimatrix}
\begin{split}
\vp=\begin{pmatrix}
0 & z^{\mfr_1} & 0 &\cdots& 0\\
0 & 0 & z^{\mfr_2} &\cdots& 0\\
\vdots & \vdots &\ddots & &\vdots\\
\vdots & \vdots & &\ddots &z^{\mfr_n}\\
0& 0 & \cdots  &0 &0
\end{pmatrix}\, dz. 
\end{split}
\end{equation}
Choose a gauge transformation $g = \exp \mu$ with $\mu$ taking values in $\mathfrak h$ and define
the Hermitian metric $H = g^2$.  We now transform the fields to unitary gauge. Using 
cylindrical coordinates $y = R \sin \psi$, $r = R \cos \psi$ with $r = |z|$, $R = |(r,y)|$, the fields become
\begin{equation}
A_2=-i\pa_3 \mu,\;A_3 = i \pa_2\mu,\;A_y=0,\;\phi_z=e^{\mu}\vp e^{-\mu},\;\phi_1=-i\pa_y\mu
\end{equation}
(recall $z = x_2 + i x_3$).    Now write $\mu=\frac{1}{2}\sum_{ij}A^{-1}_{ij}H_i\zeta_j$ for some functions $\zeta_j$. Then
\eqref{EBE} reduces to the system 
\begin{equation}
\sum_j A_{sj}^{-1}\Delta_{3} \psi_j+r^{2\mfr_s+1}e^{\psi_s}=0,\ \ s = 1, \ldots, n,
\end{equation}
where $\Delta_{3}=-(\pa_{x_2}^2+\pa_{x_3}^2+\pa_y^2)$.  Changing variables $\zeta_j:=q_j-2(\mfr_j+1)\log r$ yields
\begin{equation}
\sum_{j}A_{sj}^{-1}r^2\Delta_{3}q_j-e^{q_s}=0,
\end{equation}
and finally, defining $\sigma$ by $y/r = \sinh \si$, i.e., $\si=\log(\frac{\sqrt{y^2+r^2}+y}{r})$, and assuming that 
$q_j = q_j(\si)$, we obtain the `repulsive' Toda system
\begin{equation}
\frac{d^2 q_i }{d\si^2}-\sum_{j}A_{ij}e^{q_j}=0.
\end{equation}
The Nahm pole boundary condition now takes the form
\begin{equation}
q_j\sim -2\log \psi +\log B_j+\cdots \ \ \mbox{as}\ \ \psi \to 0.
\end{equation}

It turns out to be more convenient to analyze these equations using the functions $\chi_i=\sum_jA_{ij}^{-1}q_j$,
and in terms of these, the model Hermitian metric equals
\begin{equation} 
\HM:=\exp(\sum_i(\chi_i-2\sum_jA_{ij}^{-1}(\mfr_j+1)\log r)H_i).
\label{modelHermitianmetricforknot}
\end{equation}
The derivation of expressions for the functions $\chi_i$ involves more intricate Lie theoretic considerations, 
for which we refer to \cite{Mikhaylov2012solutions}.  This expression for $\HM$ leads
to (unfortunately quite complicated) expressions for the model solution fields $(\AM,\PM,\POM)$ in unitary gauge. 

\begin{theorem}[\cite{witten2011fivebranes}, \cite{Mikhaylov2012solutions}] 
There exists a model knot solution on $\CC \ti \RP$ with t'Hooft singularity at $z=0$. It can be given either in 
terms of the Hermitian metric $\HM$, or else in terms of the fields $\AM,\PM,\POM$ in unitary gauge. These all satisfy 
the Nahm pole boundary conditions at $\psi=0$ (i.e., as $y \to 0$ for $R > 0$) and  
are homogeneous of degree $-1$ in $R$, so in particular
$|\AM| = \MO(R^{-1})$, $|\PM|, |\POM| =\MO(R^{-1}\psi^{-1}),$ as $R\to 0$, $\psi \to 0$. 
\end{theorem}

We now illustrate by presenting the cases $n=1, 2$, following \cite[Appendix]{Mikhaylov2012solutions}. In the following, 
write $\chi(\si)=\sum_{i=1}^n\chi_i(\si)H_i$. 

\medskip

\noindent{\bf The model solution for $\SL(2,\CC)$.}  Here the weight is a single number $\mfr$
and $\chi_1(\si)=-\log(\frac{\sinh( (\mfr + 1)\si)}{\mfr+1})$. 

Writing $u:=\chi-(\mfr+1)\log r=\log \frac{2(\mfr+1)}{(R+y)^{\mfr+1}-(R-y)^{\mfr+1}}$, then %the model Hermitian metric becomes
$$
\HM=\exp((\chi-(\mfr+1)\log r) H)=\begin{pmatrix}     e^{u} & 0 \\     0 & e^{-u}     \end{pmatrix}.
$$
%Using \eqref{explicitmodelunitarygauge}, then we obtain the expressions 
and in unitary gauge
\begin{equation*}
\begin{split}
\phi_z=&\frac{1}{r} \exp(i\mfr\theta+\chi(\si))\begin{pmatrix}
0 & 1 \\
0 & 0
\end{pmatrix}=\frac{2(\mfr+1)e^{i\mfr\theta}\cos^\mfr \Psi}{R(1+\sin \Psi)^{\mfr+1}-(1-\sin \Psi)^{\mfr+1}}\begin{pmatrix}
0 & 1 \\
0 & 0
\end{pmatrix},\\[5mm]
\phi_1=&-\frac{i}{2R}\pa_\si \chi(\si)=\frac{n+1}{R}\frac{(1+\sin \Psi)^{\mfr+1}+(1-\sin \Psi)^{\mfr+1}}{(1+\sin \Psi)^{\mfr+1}-(1-\sin \Psi)^{\mfr+1}}\begin{pmatrix}
\frac i2 & 0 \\
0 & -\frac i2
\end{pmatrix},\\[5mm]
A=&  % -(\mfr+1)\begin{pmatrix}
% \frac i2 & 0 \\
% 0 & -\frac i2
% \end{pmatrix}+\frac{1}{2}\sin \Psi \frac{(1+\sin \Psi)^{\mfr+1}+(1-\sin \Psi)^{\mfr+1}}{(1+\sin \Psi)^{\mfr+1}-(1-\sin \Psi)^{\mfr+1}}d\theta \begin{pmatrix}
% \frac i2 & 0 \\
% 0 & -\frac i2
% \end{pmatrix}\\[5mm]
-(\mfr+1)\cos^2 s\frac{(1+\sin \Psi)^{\mfr}-(1-\sin \Psi)^{\mfr}}{(1+\sin \Psi)^{\mfr+1}-(1-\sin \Psi)^{\mfr+1}}d\theta \begin{pmatrix}
\frac i2 & 0 \\ 0 & -\frac i2
\end{pmatrix}.
\end{split}
\end{equation*}
These formul\ae\ appeared in \cite{MazzeoWitten2017} and were used in \cite{HeMazzeo2017}.

\medskip

\noindent{\bf The model solution for $\SL(3,\CC)$.} When $n=2$, we set $m_i=\mfr_i+1$, $i=1,2$.  The solutions
to the Toda system are now given by 
\begin{equation*}
\begin{split}
\exp(-\chi_1)&=\frac{1}{4}\left(\frac{\exp(\frac{4m_1+2m_2}{3}\si)}{m_1(m_1+m_2)}-\frac{\exp(\frac{-2m_1+2m_2}{3}\si)}{m_1m_2}+
\frac{\exp(-\frac{2m_1+4m_2}{3}\si)}{m_2(m_1+m_2)}\right),\\[0.5ex]
\exp(-\chi_2)&=\frac{1}{4}\left(\frac{\exp(\frac{2m_1+4m_2}{3}\si)}{m_2(m_1+m_2)}-
\frac{\exp(\frac{2m_1-2m_2}{3}\si)}{m_1m_2}+\frac{\exp(-\frac{4m_1+2m_2}{3}\si)}{m_1(m_1+m_2)}\right),
\end{split}
\end{equation*}
while 
\[
\HM:=\exp(\sum_{=1}^2(\chi_i-2\sum_jA_{ij}^{-1} m_j \log r)H_i).
\]
We do not write out the lengthier formul\ae\ for $A$, $\phi_z$ and $\phi_1$ in this case.

\medskip

We can now use local  coordinates $(z,y)$ to transport this model solution to be centered at any point $(p,0) \in \Si \times \RP$.
This gives the approximate solution $(\AM,\PM,\POM)$ near this point with weight vector $\vfr:=\{\mfr_1,\cdots,\mfr_n\}$. 
\begin{definition}
A solution $(A,\phi,\phi_1)$ to the \EBE satisfies the Nahm pole boundary condition with knot singularity of weight
$\vfr$ at $(p,0) \in \Si \times \RP$ if, in some gauge, 
\begin{equation}
(A,\phi,\phi_1)=(\AM,\PM,\POM)+\MO(R^{-1+\ep} \psi^{-1 + \ep})
\end{equation}
for some $\ep>0$. %, where $R$ and $\Psi$ are the spherical coordinates used earlier. 
\end{definition}

\begin{definition}
In local coordinates $(z,y)$ near $(p,0) \in \Si \ti \RP$ and some local frame for $\ME$, write $\vp=\sum_i z^{\mfr_i}E_i^+$ 
and let $\HM$ be the corresponding model Hermitian metric \eqref{modelHermitianmetricforknot}.  Then another 
Hermitian metric $H=H_0e^s$ satisfies the Nahm pole boundary condition with knot singularity of weight $\vfr$ at $p$
if $|s| + |y\, ds| \leq C \psi^{\ep} R^{\ep}$. 
\end{definition}

Much as for the simpler Nahm pole boundary condition, we can also consider this in parallel 
holomorphic gauge. Suppose the Higgs field takes the form
\begin{equation}
\label{vpformNahmpole}
\begin{split}
\vp =\begin{pmatrix}
\st & z^{\mfr_1} & 0 &\cdots& 0\\
\st & \st & \mfr_2 &\cdots& 0\\
\vdots & \vdots &\ddots & &\vdots\\
\vdots & \vdots & &\ddots &\mfr_n\\
\st& \st & \cdots  &\st &\st
\end{pmatrix}\, dz.
\end{split}
\end{equation}
As before, we have a commuting triplet $(\MD_1, \MD_2, \MD_3 = \del_y)$ coming from the holomorphic
structure on $\ME$ and this $\vp$.  We then use the singular Hermitian metric $\HM$
and complex gauge transformation $g_{\mathrm{mod}} = \HM^{1/2}$ to transform this triple into one satisfying the Nahm
pole boundary condition away from $(0,0)$, and such that the transformed fields satisfy the Nahm pole
condition with knot singularity at the origin.    This is discussed further in Section 7.

We conclude this section with some useful estimates for 
\[
\HM = \exp(\mu)=\exp(\sum_{ij}A_{ij}^{-1}H_i\zeta_j) = \exp(2 \sum f_kH_k)
\]
where $f_k:=\frac{1}{2}\sum_jA_{jk}^{-1}\zeta_j$. Recall that $\mu$ and hence $\HM = \mathrm{diag}(\lam_1, \ldots, \lam_{n+1})$ 
are diagonal, where $\lam_k = \exp(f_k-f_{k-1})$. 

\begin{lemma}
\label{modelmetricestimate}
For fixed $r_0 > 0$, and every $k$, $|\lam_{k+1}\lam_k^{-1}|_{\MC^0(B_{r_0}(p,0))}\leq C$% in the half-ball 
%$B_{r_0}(p,0) \subset \mathbb{C}\ti \RP$.  
\end{lemma}
\begin{proof}
Since $\lam_{k+1}\lam_k^{-1}=\exp(f_{k-1}+f_{k+1}-2f_k)$, it suffices to show that $f_{k-1}+f_{k+1}-2f_k$ is bounded above. 
We compute 
\[
f_{k-1}+f_{k+1}-2f_k=\sum_{j}(A_{k-1,j}^{-1}+A_{k+1,j}^{-1}-2A_{kj}^{-1})\zeta_j.
\]
Using the explicit formula $A_{ij}^{-1}=\min\{i,j\}-\frac{ij}{n+1}$, we see that if $j\neq k$, then $2A_{kj}^{-1}-A_{k-1j}^{-1}-
A_{k+1j}^{-1}=0$, while if $j=k$, then $2A_{kj}^{-1}-A_{k-1j}^{-1}-A_{k+1j}^{-1}=1$. These give 
$f_{k-1}+f_{k+1}-2f_k=-q_k+(\mfr_k+1)\log r$.   Since the $q_k$ depend only on $\si$ we obtain $\lim_{R\to 0}|\lam_{k+1}\lam_k^{-1}|=0$,
while as $\si\to 0$, $q_k\sim -2\log \si+\log B_j$, so $\lim_{\psi\to 0}|\lam_{k+1}\lam_k^{-1}|=0$ too. This proves the claim.
\end{proof}

\end{subsection}
\end{section}

\begin{section}{Holomorphic Data From the Singular Boundary Conditions}
Following the program laid out in \cite{gaiotto2012knot} and explained in Section \ref{holdata}, we know that for any solution 
to the \EBE over $\Si\ti\RP_y$, there is a stable Higgs pair $(\ME_y,\vp_y)$ on each slice $\Si\ti\{y\}$ as well as a parallel 
transform $\MD_3$ which identifies the Higgs pairs in each slice with one another.  In this section we explain how 
imposing the singular boundary condition at $y=0$ yields a holomorphic line bundle $L \subset \ME$. In other 
words, a solution to the \EBE satisfying the Nahm pole boundary condition determines a triplet $(\ME, \vp,L)$. 

\begin{subsection}{Data from the Nahm Pole Boundary Condition}
Suppose $(\ME, A,\phi, \phi_1)$ satisfies the \EBE as well as the Nahm pole boundary conditions at $y=0$. By the discussion
in Section \ref{holdata}, we obtain a stable Higgs pair $(\ME,\vp)$ and a parallel transport $\MD_3=\pa_y+\MA_y$.
In addition, in a suitable trivialization near $y=0$, these fields satisfy the Nahm pole boundary conditions (without knots):
\begin{equation}
\begin{split}
&\phi_z= \frac{\mfe^+}{y}+\MO(y^{-1+\ep})=\frac{1}{y} \begin{pmatrix}
0 & \sqrt{B_1} & 0  &\cdots& 0\\
0 & 0 & \sqrt{B_2}  &\cdots& 0\\
\vdots & \vdots &\ddots &  &\vdots\\
\vdots & \vdots & &\ddots  &\sqrt{B_n}\\
0& 0 & \cdots & 0 &0 \end{pmatrix}+\MO(y^{-1+\ep}), \ \mbox{and}\\[0.75ex]
&\MA_y =  -i\phi_1 =\frac{\mfe^0}{2y}+\MO(y^{-1+\ep})= \frac{1}{2y} \begin{pmatrix}  \tfrac{n}{2} & 0 & \ldots & 0 \\ 0 & \tfrac{n-2}{2} & \ldots & 0 \\
0 & 0 & \ddots & 0 \\ 0 & 0 & \ldots  & -\tfrac{n}{2}  \end{pmatrix}+\MO(y^{-1+\ep}).
\end{split}
\label{NPdatatrivilasection}
\end{equation}	

Using coordinates associated to a local holomorphic frame $s_1, \ldots, s_{n+1}$, so $s_i$ corresponds to  the 
vector $(0,\cdots, 0, 1, 0,\cdots ,0)^\da$, and we write $\vp$ and $\phi_1$ as in \eqref{NPdatatrivilasection}.
If $\MD_3 s :=\pa_ys-i\phi_1s=0$ and $s(y)|_{y=1}=\sum_{i=1}^{n+1}{a_i}s_i$, $a_i \in \CC$, then 
\begin{equation*}
s(y) = \sum_{i=1}^{n+1} (a_i y^{-\frac{n}{2}+i-1}+ \MO(y^{-\frac{n}{2} + i -1 + \ep}))s_i\ \ \mbox{as}\ \ y\to 0.
\end{equation*}
The span of the section $s_{n+1}$ is distinguished because its parallel transport vanishes at the maximal
possible rate, $y^{n/2}$, as $y \to 0$.  This leads to the invariant description of the {\em vanishing line bundle}
\begin{equation*}
L:=\{s\in\Gamma(E): \MD_3s=0,\;\lim_{y\to 0}|y^{-\frac{n}{2}+\al}s|=0\}
\end{equation*}
for any $0 < \alpha < 1$. Since $L$ is spanned locally by $s_{n+1}$, it is a holomorphic line bundle.

Using the explicit form of $\vp = \phi_z dz$ and the fact that $\phi_z$ is nonvanishing, if we set $L_0=L$ and define 
$L_{j+1}=\vp(L_j)\otimes K^{-1}$, $j\leq n-1$, then each $L_j$ is a line subbundle of $E$ isomorphic to  $L_0$ and these
are all independent of one another. Thus, since $\det(E)=\MO_\Si$, the map 
$$
f_n:=1\we \vp\we\cdots \we \vp^n:L^{n+1}\to K^{\frac{n(n+1)}{2}}
$$
is everywhere nonvanishing, hence $L\cong K^{\frac n2}$. Slightly more generally, define the flag 
$E_0 = L \subset E_1 \subset \ldots \subset E_n = \ME_{y=1}$, where for each $j$, the parallel transport of sections of $E_j$ 
vanish at least like $y^{n/2 - j}$ as $y \to 0$. The expressions for $\MA_y$ and $\vp_z$ above show that 
$\vp(E_j)\subset E_{j+1}$, $j = 0, \ldots, n-1$, and furthermore, $L_j \subset E_{j}$ and $\vp(E_j)\otimes K^{-1} \subset E$, hence
\begin{equation}
E_j = \bigoplus_{i=0}^{j} L_i %= & (\vp(L_{j -1})\otimes K^{-1}) \oplus \cdots \oplus (\vp(L)\otimes K^{-1}) \oplus L \\
= K^{n/2 - j + 1} \oplus \ldots \oplus K^{n/2}.
\label{directsumE}
\end{equation}

The conditions that $f_n$ is everywhere nonvanishing is a strong restriction: 
\begin{proposition}
If $(\ME,\vp)$ is a stable Higgs pair with a line subbundle $L \subset \ME$, and if the holomorphic map 
$f_n=1\we \vp\we\cdots \we \vp^n:L^{n+1}\to K^{\frac{n(n+1)}{2}}$ has no zeroes, then $L\cong K^{\frac{n}{2}}$ and 
$(\ME,\vp)$ lies in the Hitchin component $\MM_{\Hit}$. 
\end{proposition}
\begin{proof}
As before, write $\vp=\phi_z dz$ in some local holomorphic chart. We have seen above that $S^{n}(K^{-\frac{1}{2}}\oplus K^{\frac{1}{2}})
=K^{-\frac{n}{2}}\oplus K^{-\frac{n}{2}+1}\oplus\cdots\oplus K^{\frac{n}{2}}$, where $L_j=K^{\frac{n}{2}-j}$. Let $e_0$ be a nonvanishing
local holomorphic section of $L$ and define $e_{i+1}=\phi_z(e_i)$. Then $\{e_0,\cdots,e_{n}\}$ is a basis of $E$, and
in the trivialization defined by this frame, the Higgs field takes the form \eqref{HitchincomponentHiggsfield}.
\end{proof}

Define the moduli space of solutions to the extended Bogomolny equations with Nahm pole boundary condition
\begin{multline*}
\MM_{\mathrm{NP}}:=\{(A,\phi,\phi_1): \ \mathrm{EBE}(A,\phi,\phi_1)=0, \  (A,\phi,\phi_1) \mbox{ converges to a flat }
\mathrm{SL}(n+1,\mathbb{C}) \\ \mbox{connection as} \ y\to\infty\ \mbox{and}\ \mbox{ satisfies the Nahm Pole boundary condition 
at} \ y=0\}/\mathcal{G}. 
\end{multline*}
This proves
\begin{proposition}{\cite{gaiotto2012knot}} \label{INP}
There is a well-defined map $I_{\mathrm{NP}}:\MM_{\mathrm{NP}}\to \MM_{\mathrm{Hit}}$.  In other words, 
let $(A,\phi,\phi_1)$ satisfy the \EBE and Nahm pole boundary conditions at $y=0$. Denote by $(\ME,\vp)$ 
the Higgs pair associated to this data and let $L \subset \ME$ be the vanishing line bundle. Then the 
the successively defined bundles $L_j$ are everywhere independent and $(\ME,\vp)$ must 
lie in the Hitchin component of $\mathrm{SL}(n+1,\mathbb{R})$ Higgs bundles, and $L = K^{n/2}$. 
\end{proposition}

Approaching this in the other direction, suppose $(\ME, \vp)$ lies in the Hitchin component, and in addition that 
the line bundle $L$ equals $K^{n/2}$, so $\ME = K^{-n/2} \oplus \ldots \oplus K^{n/2}$ and $\vp$ is as 
in \eqref{HitchincomponentHiggsfield}.  It is straightforward that regardless of the values of $q_2, \ldots, q_{n+1}$, 
the bundles $L$, $\vp(L)$, \ldots, $\vp^n(L)$ are pointwise 
independent. By virtue of the particular structure of this holomorphic bundle, these are all subbundles of $\ME$. As in
\cite{gaiotto2012knot}, we express this by saying that $L \we \vp(L) \we \ldots \we \vp^n(L)$ is nonvanishing.  
As we now explain, this determines an approximate (and later, an exact) solution of the \EBE which 
satisfies the Nahm pole boundary condition without knots at $y=0$. 

Proceeding now as in Section 3.2 (just after Definition 3.3), we may conjugate using precisely the same complex
gauge transformation $g_0$ as there to obtain fields satisfying the Nahm pole boundary conditions
\begin{equation}
\phi_z= \frac{\mfe^+}{y}+\MO(y^{-1+\ep}), \ \ \phi_1:=\frac{i\mfe^0}{2y}+\MO(y^{-1+\ep}),\ \ \mbox{as}\ \ y \to 0.
\label{modhf}
\end{equation}
\end{subsection}

\begin{subsection}{Data from the Nahm Pole Solution with Knot Singularities}
We have described the situation when the line bundles $L_j$ determined by $\vp$ and the vanishing
line bundle $L_0$ are everywhere independent. In general this is not the case, and in fact the dependency locus
determines the locations and orders of the knot singularities at $y=0$.  

To understand this, we first examine parallel transport near $y=0$ for the model knot solution. 
Using the convention in Section \ref{subsectionknotsingularity}, we write $\HM = g_0^2$ where $g_0=
\exp(\frac{1}{2}\sum_{ij}A_{ij}^{-1}H_i\zeta_j)$ is the {\em diagonal} matrix $\diag(\lam_1,\cdots,\lam_{n+1})$,
where $\lam_i\sim {\psi}^{-\frac{n}{2}+i-1}$ as $\psi \to 0$.  
%As before, we are particularly interested in sections vanishing like $\psi^{n/2}$. 

Now, $\zeta_j=q_j(\psi)-(\mfr_j+1)\log r$ (since $\psi$ is a function of $\si$) and 
$q_j = -2\log \psi+\MO(1)$, so 
\begin{equation}
\begin{split}
\lam_{n+1} & =\exp(-\frac{1}{2}\sum_jA_{nj}^{-1}\zeta_j) \\
& =\exp(-\frac{1}{2}\sum_jA_{nj}^{-1}q_j+\sum_jA_{nj}^{-1}(\mfr_j+1)\log r))  = \MO(\psi^{\frac{n}{2}}r^{1+\ep}) %
\end{split}
\end{equation}

Fix a local holomorphic frame $s_1, \ldots, s_{n+1}$, with $s_i$ corresponding to $(0,\cdots, 0, 1, 0,\cdots ,0)^\da$ in 
the same trivilization as in \eqref{cannonicalformsl2Liealgebra}.  Suppose $\MD_3 s := g_0 \del_y g_0(y)^{-1} s = 0$ 
and $s=\sum_{i=1}^{n+1}a_i(y) s_i$. Then $\sum \lambda_i(y)^{-1} a_i(y) s_i = \sum \lambda_i(1)^{-1} a_i(1) s_i$, so 
\begin{equation}
s(y)= \sum_{i=1}^{n+1}\lam_i(y) \lambda_i(1)^{-1} a_i(1) s_i.
\label{modelleadingbehavior}
\end{equation}
In general \eqref{modelleadingbehavior} is the leading behavior for any solution with knot singularity,
so in this case too we can define the vanishing line bundle $L$ as the span of sections which vanish at the maximal 
rate $\psi^{\frac{n}{2}}$, $\psi\to 0$.  

More explicitly, if the knot singularities occur at $p_1,\cdots,p_k$, then let $\calU_j$ be a collection of open disks, with
$\calU_j$ centered at $p_j$, and $\calU_0$ an open set not containing any $p_j$ such that $\cup_{j=0}^{k} \calU_j=\Si$. 
Using cylindrical coordinates $(R_j,\theta_j,\psi_j)$ in $\calU_j \ti \RP$, define 
\begin{equation}
\begin{split}
L:=\{s: \MD_3s=0,\;& \lim_{y\to 0}|y^{-\frac{n}{2}+\al}s|=0\ \mbox{on} \ \calU_0\ \\
& \mbox{and}\ \lim_{\psi_j\to 0}|\psi_j^{-\frac{n}{2}+\al}s|=0\ \mbox{on} \ \calU_j, \ j = 1, \ldots, k\},
\end{split}
\end{equation}
which is a holomorphic line bundle.  As in \eqref{varphimatrix}, for each $j$,  $L$ determines an $n$-tuple of 
positive integers $\vfr=\{\mfr_1,\cdots,\mfr_n\}$ (we omit the label $j$ for simplicity) which encode 
the dependency relationships between the iterates $\vp^i(L) K^{-i}$ as follows.  Near each $p_j$, the (local) 
holomorphic map $f_1 = 1 \we \vp :L^2 \to K\otimes\we^2E$ vanishes at $p_j$, with vanishing order $Z(f_1)$;
set $\mfr_n:=Z(f_1)$. Similarity, for $2 \leq i \leq  n$, $f_i = 1\we\vp\we\cdots\vp^i $ 
is a local holomorphic section of $K^{\frac{(i+1)i}{2}}\otimes L^{-(i+1)}\otimes \we^{i+1}E$ which vanishes to order
$Z(f_j)$ at $p_j$, and we set $\mfr_{n-i}:=Z(f_i)-Z(f_{i-1})$. Clearly $\mfr_{i}\geq0$ for all $i$. 

On the other hand, for any triple $(\ME,\vp,L)$, we may define a divisor $\mathfrak d = \mathfrak d(L,\vp)$ 
via the zeroes of $f_i$. We call the triple $(\ME,\vp,L)$ or the data set $\mathfrak d = \mathfrak d(L,\vp)$ effective if 
each $Z(f_i)-Z(f_{i-1}) \geq 0$.

%  general, given a triple $(\ME,\vp,L)$, we can define the corresponding section $T_i\in \Gamma(K^{\frac{(i+1)i}{2}}\otimes L^{-(i+1)})$ and the zeroes of $T_i$ will provide us the position of the knot and the corresponding integers. Let $Z(T_i)$ to be the divisor corresponding to the zeroes of $T_i$, as we only consider the integer ${\mfr_i}$ to be  non-negative, in other worInd, we wish $Z(T_1)$ and $Z(T_i)-Z(T_{i-1})$ be effective divisors. 
% \begin{definition}
% Under the previous notation, we call a triple $(\ME,\vp,L)$ effective if the corresponding divisors $Z(T_1)$ and $Z(T_i)-Z(T_{i-1})$ are effective divisors.
% \end{definition}

We also require that as $y\to\infty$, the solution converges to an irreducible flat $SL(n+1,\mathbb{C})$ connection.
This means that $(\ME,\vp)$ must be a stable Higgs pair. 

These considerations lead us to define the space of holomorphic data 
\begin{equation}
\begin{split}
\MCK :=\{(\ME,\vp,L): \mathfrak d(L,\vp) \ \mbox{is effective and}\ 
(\ME,\vp) \ \mbox{is stable}\}/\MGC,
\end{split}
\end{equation}
and the space of solutions of the \EBE with generalized knot singularities
\begin{equation}
\begin{split}
\MM_{\mathrm{GNP}} & :=\left\{ (A,\phi,\phi_1): \mathrm{EBE}(A,\phi,\phi_1)=0,\ (A,\phi,\phi_1)\ \textrm{ satisfies the Nahm pole } \right. \\
& \textrm{boundary condition with knot singularities and converges to a flat} \\ & \left. \mathrm{SL}(n+1,\mathbb{C}) 
\ \textrm{connection as} \ y\to\infty \right\}/\mathcal{G}. 
\label{complexgeometrymodulispace}
\end{split}
\end{equation}

\begin{proposition}{\cite{gaiotto2012knot}}
\label{IGNP}
Any solution $(A,\phi,\phi_1)$ to the \EBE which satisfies Nahm pole boundary conditions with knot singularities
determines an effective triple $(\ME,\vp,L)$, and there is a well-defined map $I_{\mathrm{GNP}}:\MM_{\mathrm{GNP}}\to\MCK$. 
\end{proposition} 
\end{subsection}

\begin{subsection}{The Gaiotto-Witten Conjecture}
We have now given the background which explains the precise conjecture of Gaiotto and Witten \cite{gaiotto2012knot}: 
\begin{conjecture}
The maps $I_{\mathrm{NP}}:\MM_{\mathrm{NP}}\to\MM_{\mathrm{Hit}}$ and $I_{\mathrm{GNP}}:\MM_{\mathrm{GNP}}\to
\MM^{\mathbb{C}}_{\mathrm{Knot}}$ are both bijective. 
\end{conjecture}
In slightly more detail, any element $(\ME,\vp)$ in the Hitchin component $\MM_{\mathrm{Hit}}$ should correspond to a solution 
of the \EBE satisfying Nahm pole boundary conditions, and this map is bijective.  Similarly, if $(\ME,\vp,L)\in
\MM^{\mathbb{C}}_{\mathrm{Knot}}$ is an effective triple with corresponding data set $\mathfrak d(L,\vp)$, 
then there should exist a solution to the \EBE satisfying the GNP boundary conditions with this knot data. 
Data in the Hitchin component corresponds to the special case $\mathfrak d(L, \vp) = \emptyset$. 

It is perhaps worth pointing out that the moduli spaces $\MM_{\Hit}$ and $\MM_{\mathrm{GNP}}$ do not see to have
a particulary nice structure. namely, there are obvious surjective maps from these to the moduli spaces
of flat $\SL(n+1,\RR)$ and $\SL(n+1,\CC)$ connections, respectively, but the fibers, which correspond to the set
of line subbundles $L \subset \ME$ are not stable as we move around the Hitchin moduli spaces.
\end{subsection}
\end{section}

\begin{section}{Linearization and Fredholm theory}
In this section, we commence with the analysis of the extended Bogomolny equations,  
beginning with the Fredholm theory for the linearized equations. 
The point of view here is that we start with a triple $(A, \vp, \phi_1)$ satisfying the complex moment
map equations $[\MD_i, \MD_j] = 0$ and then seek to adjust the Hermitian metric by a complex gauge transformation
so that the final moment map condition 
\begin{equation}
\Omega_H:=\frac{i}{2}\Lambda([\MD_1, \MD_1^{\da}]+[\MD_2,\MD_2^{\da}])+[\MD_3,\MD_3^{\da}] = 0
\label{generalizaedcurvature}
\end{equation}
is satisfied. Here $\Lambda$ denotes contraction with the K\"ahler form. 

\begin{subsection}{Linearization}
The first step is to compute the linearization of \eqref{generalizaedcurvature}. 
\begin{proposition}  Define $H = H_0 e^s$.   Then
\begin{equation}
\Omega_H=\Omega_{H_0}+\gamma(-s)\ML_{H_0}s+Q(s),
\label{quasilinearformequation}
\end{equation}
where 
\[
\ML_{H_0}s:=\frac{i}{2}\Lambda(\MD_1\MD_1^{\da_{H_0}}+\MD_2\MD_2^{\da_{H_0}}) s +\MD_3\MD_3^{\da_{H_0}}s,
\]
and 
$$
Q(s):=\frac{i}{2}\Lambda(\MD_1(\gamma(-s))\MD_1^{\da_{H_0}}s+\MD_2(\gamma(-s))\MD_2^{\da_{H_0}}s)+\MD_3\gamma(-s)\MD_3^{\da_{H_0}}s.
$$
In these formul\ae, $\MD_i^{\da_{H_0}}$ is the formal adjoint of $\MD_i$ with respect to $H_0$ (as described just before \eqref{EBE}), 
and $\gamma(s):=\frac{e^{\mathrm{ad}_s}-1}{\mathrm{ad}_s}$. 
Furthermore, 
\begin{equation}
\begin{split}
\langle \Omega_{H}-\Omega_{H_0},s\rangle _{H_0}=\Delta|s|_{H_0}^2+\frac{1}{2}\sum_{i=1}^3|v(s)\na_i s|_{H_0}^2,
\label{Keyequation2}
\end{split}
\end{equation}
where $v(s)=\sqrt{\gamma(-s)}=\sqrt{\frac{1-e^{-ad_s}}{ad_s}}$, $\Delta=\Delta_{\Sigma}-\pa_y^2$, $\na_i=\MD_i+\MD_i^{\da}$ for $i=1,2$ and $|v(s)\na_3 s|^2=|v(s)\MD_3s|^2+|v(s)\MD_3^{\da}s|^2$.
% and $|.|_{H_0}$ is the norm define by the admissible Hermitian metric $H_0$.
\end{proposition}
\begin{proof}
By definition, $\MD_i^{\da_H} = \MD_i^{\da_{H_0}} + e^{-s}(\MD_i^{\da_{H_0}}e^s)$ and $\vp^{\dag_H}=\vp^{\dag_{H_0}}+e^{-s}[\vp^{\dag_{H_0}},e^s]$. 
If $w \mapsto X(w)$ is any smooth family of Hermitian matrices, then 
\begin{equation}
\pa_w e^{X}=e^{X}\gamma(-X)\pa_w X =\gamma(X)\pa_w X e^{X}.
\label{Exponentialderivative}
\end{equation}
Here $\del_w$ is a `generic' derivative with respect to the parameter, and could be one of the operators $D_i$ or $D_i^{\dag_H}$,
for example.  Using this we have
\begin{equation}
\begin{split}
\Omega_H %&=\frac{i}{2}\Lambda([\MD_1,\MD_1^{\da_H}]+[\MD_2,\MD_2^{\da_H}])+[\MD_3,\MD_3^{\da_H}]\\
&=\Omega_{H_0}+\frac{i}{2}\Lambda(\MD_1(e^{-s}\MD_1^{\da_{H_0}}e^s)+\MD_2(e^{-s}\MD_2^{\da_{H_0}}e^s))+
\MD_3(e^{-s}\MD_3^{\da_{H_0}}e^s)\\
&=\Omega_{H_0}+\frac{i}{2}\Lambda(\MD_1(\gamma(-s)\MD_1^{\da_{H_0}}s)+\MD_2(\gamma(-s)\MD_2^{\da_{H_0}}s))+
\MD_3(\gamma(-s)\MD_3^{\da_{H_0}}s)\\ &=\Omega_{H_0}+\gamma(-s)\ML_{H_0}s+Q(s),
\end{split}
\end{equation}
as asserted.

For (b), first write
\[
\langle \Omega_H-\Omega_{H_0},s\rangle _{H_0}=\langle \frac{i}{2}\Lambda(\MD_1(\gamma(-s)\MD_1^{\da,H_0}s)+
\MD_2(\gamma(-s)\MD_2^{\da,H_0}s))+\MD_3(\gamma(-s)\MD_3^{\da,H_0}s),s\rangle _{H_0}.
\]
The first term equals
\begin{equation*}
\begin{split}
\lan\frac{i}{2}\Lambda\MD_1(\gamma(-s)\MD_1^{\da_{H_0}}s),s\ran&=i\Lam\bar{\pa}\lan\gamma(-s)\MD_1^{\da_{H_0}}s,s 
\ran+\lan \gamma(-s)\MD_1^{\da_{H_0}}s,\MD_1^{\da_{H_0}}s\ran\\
&=i\Lambda \bar{\pa}\lan\MD_1^{\da_{H_0}}s,\gamma(-s)s \ran+|v(-s)\MD_1^{\da_{H_0}}s|^2\\
&=2i\Lambda \bar{\pa}\pa|s|^2+|v(-s)\MD_1^{\da_{H_0}}s|^2\\
&=-\Delta_{\Sigma}|s|^2+\frac{1}{2}|v(-s)\MD_1^{\da_{H_0}}s|^2.
\end{split}
\end{equation*}
To justify these steps, note we use in the first equality that $\MD_1 = 2\pab$; the sign of the second term on the right comes 
from the fact that $\MD_1^{\dag_{H_0}}$ is a $1$-form, while in the third equality we use that $\mathrm{ad}(s) (s) = 0$, so 
$\gamma(-s) s = s$.  The second term becomes 
\[
\langle\frac{i}{2} \Lambda [\vp, \gamma(-s) [\vp^{\dag_{H_0}}, s]],s\rangle =\langle \gamma(-s)[\vp^{\dag_{H_0}},s],[\vp^{\dag_{H_0}},s]\rangle 
=|v(-s)\MD_2^{\da_{H_0}}s|^2.
\]
Finally, calculating as for the first term, the third term equals
\begin{equation*}
\begin{split}
\lan \MD_3(\gamma(-s)\MD_3^{\da_{H_0}}s),s\ran&=\pa_y\lan \gamma(-s)\MD_3^{\da_{H_0}}s,s\ran+ 
\lan \gamma(-s)\MD_3^{\da_{H_0}}s,\MD_3^{\da_{H_0}}s\ran\\ &=-\pa_y^2|s|^2+\frac{1}{2}|v(-s)\na_3 s|^2.
\end{split}
\end{equation*}
Combing all these computations yields the desired identity.
\end{proof}

Recall Simpson's K\"ahler identities \cite[Lemma 3.1]{Simpson1988Construction}: 
\begin{equation}
i[\Lambda,\MD_i]=(\MD_i^{\da_H})^{\st},\ \ i[\Lambda,\MD_i^{\da_H}]=-(\MD_i)^{\st},\ \ i = 1, 2. 
\end{equation}
To be clear, the adjoint here is taken with respect to the usual inner product on forms (as opposed to $\MD_i^\dag$), so
\[
\int \lan\MD_i^{\dag_{H}} \alpha \we \star \overline{\beta} \ran= \int \lan \alpha \we \overline{( \MD_i^{\dag_H})^{\dag_H} \st \beta}\ran
= \pm \int \lan\alpha \we \st ( \overline{\st ( \MD_i^{\dag_H})^{\dag_H} \st \beta}\ran,
\]
hence $(\MD_i^{\dag_H})^{\st} = \pm \st \MD_i \st$.  We also write $(\MD_3^{\da_{H_0}})^{\st}=\MD_3$.  This leads to the 
\begin{corollary}
$\ML_{H_0}=\frac{1}{2}((\MD_1^{\da_{H_0}})^{\st}\MD_1^{\da_{H_0}}+(\MD_2^{\da_{H_0}})^{\st}\MD_2^{\da_{H_0}})+ \MD_3%(\MD_3^{\da_{H_0}})^{\st}
\MD_3^{\da_{H_0}}$.
\end{corollary}

We now establish some Weitzenbock formul\ae. In the following, the various adjoints are taken with respect to any 
fixed Hermitian metric, but for simplicity we omit the metric subscripts. 
\begin{proposition}
Setting $\nabla_1:=\MD_1+\MD_1^{\da}$ and $\nabla_2=\varphi+\varphi^{\da}$, then we have
\begin{equation}
(\MD_1^{\da})^{\st}\MD_1^{\da}=\frac{1}{2}\na_1^{\st}\na_1+\frac{i}{2}\Lambda[\MD_1, \MD_1^{\da}], \qquad \;(\MD_1)^\st\MD_1
=\frac{1}{2}\na_1^{\st}\na_1-\frac{1}{2}\Lambda[\MD_1,\MD_1^{\da}]
\label{W1}
\end{equation} 
and
\begin{equation}
(\MD_2^{\da})^{\st}\MD_2^{\da}=\frac{1}{2}\na_2^{\st}\na_2+\frac{i}{2}\Lambda[\MD_2,\MD_2^{\da}],\qquad \;(\MD_2)^\st\MD_2
=\frac{1}{2}\na_2^{\st}\na_2-\frac{1}{2}\Lambda[\MD_2,\MD_2^{\da}],
\label{W2}
\end{equation}
and furthermore, 
\begin{equation}
\nabla_2^{\st}\nabla_2 s=i\Lam([\varphi,[\varphi^{\da},s]]-[\varphi^{\da},[\varphi,s]]). 
\label{W21}
\end{equation}
In addition, %For the third operator $\MD_3=\na_y-i\phi_1$, we have
\begin{equation}
\MD_3\MD_3^{\da}=-(\na_y^2+\phi_1^2)+\frac{1}{2}[\MD_3,\MD_3^{\da}],\;\MD_3^{\da}\MD_3=-(\na_y^2+\phi_1^2)-
\frac{1}{2}[\MD_3,\MD_3^{\da}].
\label{W3}
\end{equation}
\label{Calpha}
\end{proposition}
\begin{proof}
We first compute
\begin{equation}
\begin{split}
\na_1^{\st}\na_1&=\MD_1^{\st}\MD_1+(\MD_1^{\da})^{\st}\MD_1^{\da}=-i\Lambda \MD_1^{\da}\MD_1+i\Lambda \MD_1\MD_1^{\st} \\
& =2i\Lambda \MD_1\MD_1^{\da}-i\Lambda [\MD_1,\MD_1^{\da}]  
%\na_1^{\st}\na_1&=\MD_1^{\st}\MD_1+(\MD_1^{\da})^{\st}\MD_1^{\da}=-i\Lambda \MD_1^{\da}\MD_1+i\Lambda \MD_1\MD_1^{\st}
=-2i\Lam \MD_1^{\da}\MD_1+i\Lambda [\MD_1,\MD_1^{\da}].
\end{split}
\end{equation}
For \eqref{W21}, the Jacob identity (for graded Lie brackets) asserts  
$$
[\vp,[\vp^{\da},s]]+[\vp^{\da},[\vp,s]]=[[\vp,\vp^{\da}],s], 
$$
hence 
\begin{equation}
\begin{split}
i\Lam\MD_2\MD_2^{\da}=\frac{i}{2}\Lam([\vp,[\vp^{\da},s]]-[\vp^{\da},[\vp,s]])+\frac{i}{2}[[\vp,\vp^{\da}],s],\\
i\Lam\MD_2^{\da}\MD_2=\frac{i}{2}\Lam([\vp^{\da},[\vp,s]]-[\vp,[\vp^{\da},s]])+\frac{i}{2}[[\vp,\vp^{\da}],s].
\end{split}
\end{equation}

For \eqref{W3}, from $\MD_3=\na_y-i\phi_1$ and $\MD_3^{\da}=-\na_y-i\phi_1$ we obtain 
$(\MD_3-\MD_3^{\da})^2=4\na_y^2,$ $(\MD_3)^2+(\MD_3^{\da})^2=2\na_y^2-2\phi_1^2$.
In addition, 
$$
(\MD_3-\MD_3^{\da})^2=\MD_3^2+(\MD_3^{\da})^2-\MD_3\MD_3^{\da}-\MD_3^{\da}\MD_3,
$$
so altogether
\begin{equation*}
\begin{split}
2(\na_y^2+\phi_1^2)=-\MD_3\MD_3^{\da}-\MD_3^{\da}\MD_3=-[\MD_3,\MD_3^{\da}]-2\MD_3^{\da}\MD_3
%2(\na_y^2+\phi_1^2)=-\MD_3\MD_3^{\da}-\MD_3^{\da}\MD_3
=[\MD_3,\MD_3^{\da}]-2\MD_3\MD_3^{\da}.
\end{split}
\end{equation*}
\end{proof}

These formul\ae\ lead to a simpler expression for $\ML_H$:
\begin{corollary}
\[
\ML_H=\frac{1}{4}(\na_1^{\st}\na_1+\na_2^{\st}\na_2)-(\MD_y^2+\phi_1^2)+\frac{1}{2}[\Omega_H,\cdot ]
\]
where $\phi_1^2=[\phi_1,[\phi_1,\;]]$. 
\label{co:lin}
\end{corollary}
\end{subsection}

\begin{subsection}{Models and Blowups}
It will be important to understand the mapping and regularity properties of 
\begin{equation}
\ML_H:=\frac{1}{4}(\na_1^{\st}\na_1+\na_2^{\st}\na_2)-(\MD_y^2+\phi_1^2),
\label{line}
\end{equation}
which is the operator from Corollary~\ref{co:lin} when $\Omega_H = 0$, and where $(A,\vp, \phi_1)$ satisfies the Nahm 
pole boundary conditions, possibly with knot singularities, at $y=0$ and converges to a flat $\mathrm{SL}(n+1,\CC)$ 
connection as $y \to \infty$. 
For simplicity we often drop the subscript $H$ from this operator. These local and global properties follow from 
a general theory which was adapted and extended to the present context in \cite{MazzeoWitten2013, MazzeoWitten2017}.   
We shall not recapitulate much of this theory here, but introduce 
a few aspects which will be helpful for understanding the analysis later in this paper. 

A key feature of $\ML$ is that it enjoys a certain approximate homogeneity near the boundary.   Namely, in the absence of 
(or away from the) knot singularities, near $y=0$, $\ML$ is modelled by its so-called normal operator
\[
N(\ML) = \Delta_{\RR^3}  -[\phi_1, [ \phi_1, \cdot ]]  + \frac{1}{2}([\phi_z^\da, [ \phi_z, \cdot]] + [\phi_z,  [\phi_z^\da, \cdot] ]),
\]
which is the linearization at the flat model $\RR^3_+$ and where $(A, \vp_z, \phi_1)$ are equal to the leading 
model terms in \eqref{modhf}.   The normal operator $N(\ML)$ represents $\ML$ in terms of its action on
the elementary functions $y^\lam$ in the sense that
\[
\ML( y^\lam s) = N(\ML) (y^\lam s_0) + \MO(y^{\lam-1}), \qquad s_0 = s|_{y=0},
\]
for any section $s$ which is smooth up to $y=0$.  In other words, $N(\ML)$ is the part of $\ML$ which is precisely 
homogeneous of degree $-2$ with respect to the dilations $(z,y) \mapsto (\lam z, \lam y)$. 

To understand the local structure of $\ML$ near a knot singularity, we follow Section 4.2 of \cite{MazzeoWitten2017},
and in particular equations (4.14--15) in that paper. The spherical coordinate expression given there for the linearization
$\ML$ at the model knot solution is
\[
N(\ML) = -\frac{\del^2\,}{\del R^2} - \frac{2}{R} \frac{\del\,}{\del R} + \frac{1}{R^2} 
\left( -\frac{\del^2\,}{\del \psi^2} - \frac{\cos \psi}{\sin \psi} \frac{\del\,}{\del \psi} + \frac{1}{\sin^2\psi} 
\nabla^*_\theta \nabla_\theta + N_S\right),
\]
where
\[
N_S = -\phi_1^2 + \frac12 \left( [ \vp_z^\dag, [ \vp_z, \cdot]] + [ \vp_z, [\vp_z^\dag, \cdot] ]\right)
\]
and $\nabla_\theta$ is differentiation with respect to $A$ in the $\theta$ direction (at $R=1$).  
All fields are the ones for the model knot solution. A computation shows that $N_S \sim 2 \mathrm{Id}/\sin^2 \psi$ 
as $\psi \to 0$. We regard $N(\ML)$ as acting on sections on the entire space $S^2_+ \ti \RP$.  

In precisely the same sense as above, this linearization at the model knot is the model for the linearization in general
in the sense that around any solution (or indeed even any admissible Hermitian metric $H$), 
\[
\ML_H( R^\lam s) = N(\ML) (R^\lam s_0) + \MO(R^{\lam-1}), \ \ \mbox{where now}\ s_0 = s|_{R=0}.
\]

To interpret this last paragraph accurately, we introduce the blowup of $\Si \ti \RP$ around a knot singularity
at $(p,0)$.  This natural operation corresponds to replacing the point $(p,0)$ with its interior unit sphere bundle, and declaring
that the spherical coordinates around this point generate the space of smooth functions.  When done at each knot singularity, 
this produces the space
\[
(\Si \ti \RP)_{\mathfrak{p}} := [\Si \ti \RP; \{p_1, \ldots, p_k\}],
\]
which is the half-cylinder $\Si \ti [0,\infty)$ blown up in this way at the points $(p_j, 0)$, $ j = 1, \ldots, k$. 
This space has $k+1$ boundary faces: the `original' face with the knot points removed, $F_{\mathrm{or}} := \Si \setminus 
\{p_1, \ldots, p_k\}$, and the $k$ hemispheres $F_j$ which are the unit sphere bundles at each $(p_j, 0)$. 

This blown up space is convenient for several reasons which should become apparent below.  For the moment,
it provides a convenient framework to assert the following:  the normal operators of $\ML$ at $F_{\mathrm{or}}$ 
and at each $F_j$ are obtained by droppng all but the leading terms in the Taylor expansion of (the coefficients of)
$\ML$ at these faces. 
\end{subsection}

\begin{subsection}{Indicial Roots}
We next define and record the values of the indicial roots of $\ML$ at each of the faces of $(\Si \ti \RP)_{\mathfrak p}$.
These indicial roots are the formal rates of growth or decay of solutions to $\ML s = 0$, and their values
are needed to determine the global mapping properties of $\ML$. We refer to the papers above, as well 
as \cite{HeMazzeo2017}, for more discussion about their significance.  The computations of these
values takes some work, which can be found in \cite{MazzeoWitten2013, MazzeoWitten2017}.

The singular structure of $\ML$, or more simply just $N(\ML)$, along the face $F_{\mathrm{or}}$ is determined
by the leading asymptotics of 
$$
\phi_z\sim y^{-1}\mfe^++\cdots,\ \ \phi_z^{\da}\sim y^{-1}\mfe^-+\cdots,\ \  \phi_1\sim \frac{i}{2y}\mfe^0+\cdots.
$$ 
By definition, the indicial roots of $\ML$ at this face are the values $\lambda$ for which $N(\ML) (y^\lam s) = 0$,
for any (locally defined) smooth section $s$ or equivalently, if $\ML (y^\lam s) = \MO(y^{\lam-1})$ in contrast to the 
expected order $\MO(y^{\lam-2})$. This is a sort of eigenvalue computation, and the solutions are those values
$\lam$ for which there is a leading order cancellation. Writing this out more carefully, one sees that $\lam$ is an indicial root if
\begin{equation}
\lam(\lam-1)s = \frac{1}{2}([\mfe^+,[\mfe^-,s]+[\mfe^-,[\mfe^+,s]])+\frac{1}{4}[\mfe^0,[\mfe^0,s]].
\end{equation}
The operator on the left is the (purely algebraic!) Casimir operator for $\slf_2$:
$$
\Delta_{\mathrm{Cas}}s:=\frac{1}{2}([\mfe^+,[\mfe^-,s]+[\mfe^-,[\mfe^+,s]])+\frac{1}{4}[\mfe^0,[\mfe^0,s]],
$$
hence $\lam$ is an indicial root for $\ML$ if and only if $\lam(\lam-1)$ is an eigenvalue for $\Delta_{\mathrm{Cas}}$. 
The Lie algebra $\mathfrak{sl}_{n+1}$ decomposes into a direct sum 
of eigenspaces $V_j$, $j = 1, \ldots, n$, for $\Delta_{\mathrm{Cas}}$, where the eigenvalue on $V_j$ is $j(j+1)$ (the index 
$j$ is called the spin of $V_j$).  The indicial roots on $V_j$ therefore consist of the values $-j, (j+1)$. 
% so the entire set of indicial roots equals $\{-n-1, -n, \ldots, -2, 1, \ldots, n\}$. 
%Therefore the values of $\lam$ for which there exists a solution to $-\lam(\lam+1)s_\lam+ \Delta_{\mathrm{Cas}}s_\lam=0$ are 
%$\lam = j,-(j+1)$. 
\begin{proposition}
The set of indicial roots of the linearized \EBE for the Nahm pole away from the knot consists of the values 
$\{-n, \dots, -1, 2,\ldots, n+1\}$. 
\end{proposition}

The indicial roots of $\ML$ near any one of the faces $F_j$ are defined in much the same way. Namely, they are 
the values $\lambda$ for which there exists a field $\Phi(\psi, \theta)$ on $S^2_+$ such that $N(\ML) R^\lam \Phi = 0$,
or equivalently, $\ML (R^\lam \Phi) = \MO( R^{\lam - 1})$ rather than the expected rate $\MO(R^{\lam-2})$.  
This is once again an eigenvalue problem, but now for a differential operator on $S^2_+$ rather than a finite dimensional
endomorphism. The calculations in this setting are carried out in \cite{MazzeoWitten2017}; strictly speaking, that paper treats
the four-dimensional KW equations, but certain of those calculations correspond to perturbations of a model
knot solution which are compatible with the three-dimensional reduction, and those are the ones of interest here.
In the terminology of that paper, these are the indicial roots of type $II$.

The operator $\Delta_S$ on the hemisphere $S^2_+$ which is the expression in parentheses above has discrete 
spectrum on $L^2$ fields (this is because of the positivity and `regular singular' blowup of $N_S$ at $\psi=0$).  
It is shown in \cite{MazzeoWitten2017} that for the group $G_\CC = \mathrm{SL(2,\CC)}$, the eigenvalues $\gamma$ 
of $\Delta_S$ are all strictly greater than $2$.  Examining that proof, however, shows that this only relies on the 
positivity of $N_S$ but not on its specific structure, so the same bound is true for $G_\CC = \mathrm{SL(n+1,\CC)}$. 
We can now carry out the calculation of indicial roots using the spherical coordinate expression for $\Delta_H$
and the eigendecomposition for $\Delta_S$. 
\begin{proposition}
The set of indicial roots of $\Delta_H$ at a knot singularity of any weight $\vec{\mathfrak{r}}$ equals
\[
\{-\frac12 \pm \sqrt{\gamma + 1/4}: \gamma\ \mbox{an eigenvalue of} \ \Delta_S \}.
\]
The bound $\gamma > 2$ implies that all indicial roots are contained in the half-lines $(-\infty, -2)$ and $(1, \infty)$. 
Only those roots greater than $-1$ are relevant for this problem, and all of these are in fact strictly greater than $+1$. 
\end{proposition}
\end{subsection}

\begin{subsection}{Function spaces and Fredholm mapping theory}
We now state the key mapping properties of $\ML$ on a family of dilation-covariant weighted H\"older spaces 
adapted to the degeneracy of this operator. 

As in the last two subsections, these spaces are defined slightly differently near the face $F_{\mathrm{or}}$ and
near each of the knot faces $F_j$. In both cases, the approximate homogeneity of $\ML$ is reflected in an 
approximate scale-invariance of the associated H\"older norms.

We first define `Whitney cubes' near each of these boundary faces.  A Whitney cube $Q^1$ is a coordinate cube of 
diameter $\epsilon/2$ centered at a point $(y_0, x_0)$ where $y_0 = \epsilon$. Let $D^1_\epsilon$ denote the dilation 
$(y,z) \to (\epsilon y, \epsilon z)$.  Similarly, a Whitney cube $Q^2$ is a coordinate cube $\{ \epsilon/2 < R < 3\epsilon/2\} 
\times Q'$ where $Q'$ is any `size $1$ cube' in the $(\theta, \psi)$ coordinates; in this region we use the dilation
$D^2_\epsilon: (R, \theta, \omega) \to (\epsilon R, \theta, \omega)$.   Consider next  the dilations of these cubes:
$\tilde{Q}^j_\epsilon = (D^j_{\epsilon})^{-1}(Q^j)$, $j = 1,2$.  Our basic scale of $\ie$ H\"older spaces $\calC^{k,\alpha}_{\ie}$ 
($\ie$ stands for `iterated edge') is determined by norms which are invariant under these dilations: 
\[
||u|_{Q^j} ||_{\ie; k, \alpha} = || (D^j_\epsilon)^* u|_{\tilde{Q}^j}||, \quad j = 1,2.
\]
This type of invariance is achieved by basing these spaces on differentiations by $y \del_y, y \del_{x_1}, y\del_{x_2}$ near
$F_{\mathrm{or}}$ and by $R\del_R, \del_\theta, \del_\psi$ near each $F_j$.  For simplicity we describe these spaces
only for functions, since the adaptations for spaces of sections of any bundle are just notational.  When there are no knot
singularities, then only the first type of dilations are used, and it is more customary to denote these spaces
by $\calC^{k,\alpha}_0$ in that case.

\begin{definition}  Define $\calC^{k,\alpha}_{\ie}( \Si \ti \RP)$ to be the space of all functions $u$ on $\Si \ti \RP$ such that
\begin{itemize}
\item[i)] Near $F_{\mathrm{or}}$,  
\[
||u||_{L^\infty} + \sup_{i + |\beta| \leq k} [ (y\del_y)^i (y\del_x)^\beta u ]_{\ie; 0,\alpha} < \infty,
\]
where in this region
\[
[v]_{\ie; 0, \alpha} := \sup_{Q^1} \sup_{(y,x) \neq (y', x')\atop 
(y,x), (y', x') \in Q^1}  \frac{|u(y,x) - u(y', x')|(y+y')^\alpha}{|y-y'|^\alpha + |x-x'|^\alpha}.
\]
\item[ii)] Near each $F_j$, 
\[
||u||_{L^\infty} + \sup_{i + p + q \leq k} [(R\del_R)^i \del_\theta^p \del_\psi^q u ]_{\ie, 0, \alpha} < \infty,
\]
where here, 
\[
[v]_{\ie, 0, \alpha} := \sup_{Q^2} \sup_{(R,\theta,\psi) \neq (R', \theta', \psi') \atop
(R,\theta,\psi), (R', \theta', \psi') \in Q^2}  \frac{|u(R,\theta, \psi) - u(R', \theta', \psi')|
(R+R')^\alpha}{(|\theta-\theta'| + |\psi-\psi'|)^\alpha(R+R')^\alpha + |R-R'|^\alpha}.
\]
\item[iii)] Away from all boundaries we simply require that $u$ lies in the ordinary H\"older space $\calC^{k,\alpha}$
on each slab $\Si \ti [L, L+1]$, uniformly for $L \geq 1$. 
\end{itemize}

Finally, if $\mu$, $\nu$ and $\delta$ are any real numbers, we define
\begin{equation*}
\begin{split} 
y^\mu \calC^{k,\alpha}_{\ie}(\Si \ti \RP) & = \{u = y^\mu v: v \in \calC^{k,\alpha}_{\ie} \} \\
\psi^\mu R^\nu e^{y\delta} \calC^{k,\alpha}_{\ie}(\Si \ti \RP) & = \{u = \psi^\mu R^\nu e^{t\delta} v: v \in \calC^{k,\alpha}_{\ie}\}.
\end{split}
\end{equation*}
Note that $\psi$ may be replaced by $y$ away from the knots. 
\label{defhold0}
\end{definition}

For simplicity, we henceforth write
\[
\calX^k_{\mu, \nu, \delta}(\Si \ti \RP; \isu(E, H_0)) = \psi^\mu R^\nu e^{y\delta} \calC^{k,\alpha}_{\ie}(\Si \ti \RP; \isu(E, H_0)).
\]

\begin{theorem}{\cite{MazzeoWitten2013, MazzeoWitten2017}} 
Let $\lam_0^\pm = -1/2 \pm \sqrt{\gamma_0 + 1/4}$ where $\gamma_0$ is the ground state eigenvalue of $\Delta_S$. Suppose 
that $\mu \in (-2,1)$, $\nu \in (\lam_0^-, \lam_0^+)$ and $\delta > 0$.  Then for any $k \geq 0$ and $0 < \alpha < 1$, 
\begin{equation}
\ML_H:  \calX^{k+2}_{\mu, \nu, -\delta}(\Si \ti \RP; \isu(E, H_0)) \longrightarrow \calX^k_{\mu-2, \nu-2, -\delta}(\Si \ti \RP; \isu(E, H_0)).
% \psi^\mu R^\nu e^{-t\delta} \calC^{k+2,\alpha}_{\ie}(\Si \ti \RP; \isu(E, H_0)) \longrightarrow
%\psi^{\mu-2} R^{\nu-2} e^{-t\delta} \calC^{k, \alpha}_{\ie}(\Si \ti \RP; \isu(E, H_0))
\label{Fred1map}
\end{equation}
is a Fredholm operator of index $0$. If $\ML_H s = 0$, then $s=0$, so that in fact \eqref{Fred1map} is an isomorphism.
\label{Fred1}
\end{theorem}
\begin{proof}  To prove that this mapping is Fredholm for this set of ranges of the weight parameters, we produce
a parametrix for $\ML_H$, which lies in the `iterated edge' pseudodifferential calculus, and then show that this parametrix
is bounded between the appropriate spaces. The construction of the parametrix relies heavily on the invertibility of
the associated `normal operator', which in this situation corresponds to the operator $\ML$ using the model fields
on a flat half-space (possibly with knot at $0$). This invertibility is, in turn, 
proved by showing that this model operator is Fredholm, has index zero, and vanishing nullspace. The last fact
relies on a linearization of a Weitzenbock formula.  All of this is carried out first away from knots in
\cite{MazzeoWitten2013} and then later near knots in \cite{MazzeoWitten2017}.  The parametrix and its mapping
properties for $y$ large rely on the more standard analysis on manifolds with cylindrical ends. 

That the index of $\ML$ is $0$ follows from the fact that this weight range contains the point of symmetry
for the indicial root set, which in turn is related to the symmetry of this operator on $L^2$.  

Finally, the injectivity of this mapping is verified by noting that if $\Delta_H s = 0$, then the regularity theory in 
these cited papers implies that $|s| \leq \psi^1 R^{\lam^+_0}e^{-t\delta}$, which in turn justifies an 
integration by parts, leading to the conclusion that $\nabla_1 s = \nabla_2 s = \MD_3 s = 0$. 
These imply at $|s| \not\to 0$ as $y \to \infty$, for example, which is a contradiction. 
\end{proof}
\end{subsection}
\end{section}

\begin{section}{The Continuity Method}
We solve the \EBE with generalized Nahm boundary conditions using a standard continuity method.  In this
section we set the problem up and discuss the (easier) `openness' part of this proof. 

\begin{subsection}{The Admissible Hermitian Metric and the Family of Equations}
Fix holomorphic data $(\ME,\vp,L)$, and consider the corresponding weight set $S(\ME,\vp,L)=\{(p,\mfr_1^p,\cdots,
\mfr_n^p)\}$. In holomorphic gauge, write $\MD_1= w\bar{\pa}_{\ME}$, $\MD_2=[\vp,\cdot ]$ and $\MD_3=\pa_y$ and
set $\Theta_0 = (\MD_1, \MD_2, \MD_3)$. Complex gauge transformations act by $\MD_i^g:=g\circ \MD_i\circ g^{-1}
=\MD_i+g(\MD_i g^{-1})$.   Recall that we can either work in a holomorphic gauge with this regular triple $\Theta_0$
and look for a singular Hermitian metric or else, by Proposition \ref{complexgaugeactionchange}, first transform
$\Theta_0$ by a singular gauge transformation and look for a solution of the \EBE which is a regular
Hermitian metric. In this latter formulation, which is the one we shall be using, the equations have singular
coefficients.

\begin{definition}
A Hermitian metric $H_0$ is admissible if the following conditions hold: 
\begin{itemize}
\item The Chern connection associated to $H_0$ has knot singularity of weight $\vfr$
at $(p,0) \in \Si \ti \RP$ for each $(p,\vfr)\in S(\ME,\vp,L)$, and satisfies the Nahm pole boundary condition elsewhere along $y=0$.
\item The Chern connection converges to the flat $\mathrm{SL}(n+1,\mathbb{C})$ connection defined by the stable Higgs pair 
$(\ME,\vp)$ as $y \to \infty$. 
\item $\Omega_{H_0}$ vanishes to infinite order at $y=0$. 
\end{itemize}
\end{definition}

If $H_0$ is admissible in this sense, denote by $i\su(E,H_0)$ the subspace of Hermitian endomorphisms in $\End(E)$.
For any $s\in isu(E,H_0)$, define the new Hermitian metric $H=H_0e^{s}$ and the family of maps
\begin{equation}
N_t(s):=\Ad(e^{\frac s2})\Omega_{H}+ts=0.
\label{eqmethodofContinuation}
\end{equation}
To parse this definition, note that $\Omega_H\in i\su(E,H)$ and $Ad(e^{\frac{s}{2}}):i\mathfrak{su}(E,H)\to i\mathfrak{su}(E,H_0)$ 
is a bundle isomorphism which satisfies
\[
\langle \Ad(e^{\frac{s}{2}})f, \Ad(e^{\frac{s}{2}})g\rangle _{H}=\langle f,g\rangle _{H_0} \ \ \mbox{for}\ \ \ 
f,g \ \ \mbox{sections of}\ \  isu(E,H). 
\]

For any small $\ep > 0$ and any $\delta > 0$, $H:=H_0e^{s}$ is an admissible metric provided 
$s\in \calX^{k+2}_{2-\ep, 1+\ep, -\delta}(\Si \ti \RP; i \mathfrak{su}(E, H_0))$, 
%\psi^{2-\ep} R^{1+\ep} e^{-t\delta}\calC^{k+2,\alpha}_{\ie}$, 
and moreover
\[
N_t : \calX^{k+2}_{2-\ep, 1+\ep, -\delta}(\Si \ti \RP; i \mathfrak{su}(E, H_0)) \longrightarrow
\calX^{k}_{-\ep, -1+\ep, -\delta}(\Si \ti \RP; i \mathfrak{su}(E, H_0))
%\psi^{2-\ep} R^{1+\ep} e^{-t\delta}\calC^{k+2,\alpha}_{\ie}(\Si \ti \RP; i \mathfrak{su}(E, H_0)) \longrightarrow
%\psi^{-\ep} R^{-1+\ep} e^{-t\delta}\calC^{k,\alpha}_{\ie}(\Si \ti \RP; i \mathfrak{su}(E, H_0))
\]
is a smooth map which also depends smoothly on $t \in [0,1]$. 

The continuity method consists in showing that the set
\begin{equation}
I:=\{t\in[0,1]:N_t(s)=0 \ \ \mbox{has a solution} \ \ s\in \calX^{k+2}_{2-\ep, 1+\ep, -\delta} \}
\label{setI}
\end{equation}

is nonempty, open and closed, so that $I = [0,1]$.   We will have solved our problem once we show that $0 \in I$. 
\end{subsection}

\begin{subsection}{Linearization}
Before proceeding further, we compute the linearization of $N_t$.   
\begin{proposition}
If $N_t(s) = 0$, then for any section $s'$ of $\isu(E,H_0)$, 
\begin{equation}
\ML_{t,s}(s'):=\frac{d}{du}|_{u=0}N_t(s+us')=Ad(e^{\frac{s}{2}})\ML_Hs'+ts';
\end{equation}
the formula on the right defines the operator $\ML_{t,s}$ on the left.
\end{proposition}
\begin{proof}
We compute that $\ML_{t,s}\, (s') $ equals
\begin{equation}
\begin{split}
%=\frac{d}{du}|_{u=0}N(s+us',t)s
& (\left.\frac{d\,}{du}\right|_{u=0}e^{\frac{s+us'}{2}})\Omega_H e^{-\frac{s}{2}}+
e^{\frac{s}{2}}\Omega_{H} (\left.\frac{d\, }{du}\right|_{u=0}e^{-\frac{s+us'}{2}})+ Ad(e^\frac{s}{2})
\left( \left. \frac{d\, }{du}\right|_{u=0}\Omega_{He^{us'}}\right)+ts' \\
& = \frac{1}{2}\gamma(\frac{s}{2})s' \Ad(e^{\frac{s}{2}})\Omega_H-\frac{1}{2}\Ad(e^{\frac{s}{2}})\Omega_H\gamma(\frac{s}{2})s'
+\Ad(e^{\frac{s}{2}}) \ML_H s' \\ 
& = \Ad(e^{\frac{s}{2}})\sum_{i=1}^3\MD_i\MD_i^{\da_H}s'+t \Ad(e^{\frac{s}{2}})s',
\end{split}
\end{equation}
where the third equality uses $\Ad(e^{\frac{s}{2}})\Omega_H+ts=0$.
\end{proof}
\end{subsection}

\begin{subsection}{$I$ is nonempty}
\begin{proposition}
	\label{Inonempty}
There exists an admissible Hermtian metric $H_0$ and a section $s\in \calX^{k+2,\alpha}_{2-\ep, 1+\ep, -\delta}$  (for any $k$)
such that $N(s,1)=0.$
\end{proposition}
\begin{proof}
For the moment we shall use a simpler definition of admissibility, that the metric satisfies the equation only
up to first order at $y = R = 0$. In the next section we show how to improve this to a solution up to 
infinite order.  

We construct $H_0$ and $s$ in two stages. For the first, let $H_{-1}$ be any metric and set $\kappa
:=\Omega_{H_{-1}}$.  By definition of admissibility, $\kappa \in \calX^{k,\alpha}_{\mu, \nu, -\delta}$ for any $\mu, \nu$
(since it vanishes to infinite order at $y=0$). Now define $H_0:=H_{-1}e^{\kappa}$.  This is certainly also admissible.
Furthermore, 
\begin{equation}
N_1(-\kappa)=\Ad(e^{-\frac{\kappa}{2}})(\Omega_{H_{-1}})-\kappa = \Ad(e^{-\frac{\kappa}{2}})(\Omega_{H_{-1}}-\kappa) = 0
\end{equation}
since $\Ad(e^{-\kappa/2}) \kappa = \kappa$.   Thus $1 \in I$, as claimed. 
\end{proof}
\end{subsection}

\begin{subsection}{$I$ is open}
We first formulate and prove a consequence of Theorem~\ref{Fred1}
\begin{proposition}
For $\ep$ sufficiently small and $\delta > 0$, 
\[
\ML_{t,s}:  \calX^{k+2,\alpha}_{2-\ep, 1+\ep, -\delta}(\Si \ti \RP; \isu(E; H_0)) \longrightarrow 
\calX^{k,\alpha}_{-\ep, -1+\ep, -\delta}(\Si \ti \RP; \isu(E; H_0))
\]
is an isomorphism. 
\label{Fred2}
\end{proposition}
\begin{proof}
The same parametrix method as in Theorem~\ref{Fred1} shows that this mapping is Fredholm and has index zero, so it suffices
to show that its nullspace vanishes. For $s\in \calX^{k+2,\alpha}_{2-\ep, 1+\ep, -\delta}(\Si \ti \RP; \isu(E,H_0))$, we compute 
\begin{equation}
\begin{split}
0 = \int\langle \ML_{t,s} s, \Ad(e^{\frac{s}{2}})s\rangle _{H_0} &
=\int\langle \Ad(e^{\frac{s}{2}})\ML_H s, \Ad(e^{\frac{s}{2}})s \rangle _{H_0}+t\langle s, \Ad(e^{\frac s2})s\rangle _{H_0}\\
&=\int\langle \sum_{i=1}^3 (\MD_i^{\da_H})^*\MD_i^{\da_{H}} s ,  s\rangle _{H}+t|\Ad(e^{\frac s4})s|^2_{H_0}\\
&=\int\sum_{i=1}^3|\MD_i^{\da_H}s|^2_{H}+t|\Ad(e^{\frac s4})s|^2_{H_0}.
\end{split}
\end{equation}
The integration by parts is justified by the decay rates of $s$. We conclude that $Ad(e^{\frac s4})=0$, hence $s = 0$, and
the operator is an isomorphism, as claimed.    

If $s_1$ and $s_2$ are sections of $\isu(E,H_0)$, then $\langle \Ad(e^{\frac{s}{2}})s_1, \Ad(e^{\frac{s}{2}}) s_2)\rangle _{H	_0}
=\langle s_1 , s_2 \rangle _{H}$, and hence 
\begin{equation}
\begin{split}
\int \langle \ML_{t,s} s_1 , s_2 \rangle_{H_0} &=\int\langle \Ad(e^{\frac{s}{2}})\ML_H s_1 , s_2 \rangle _{H_0}+\langle t s_1 , s_2 
\rangle _{H_0}\\ &=\int \langle \ML_H s_1 , \Ad(e^{-\frac s2}) s_2 \rangle _H+\langle s_1 ,t s_2 \rangle _{H_0}\\
&=\int \langle s_1, \Ad(e^{\frac{s}{2}})\ML_{t,s} (\ Ad(e^{-\frac{s}{2}})s_2)\rangle_{H_0} .
\end{split}
\end{equation}
This shows that the range is dense, since if $s_2$ is orthogonal to every $\ML_{t,s} s_1$, then $s_2 \equiv 0$. Thus it is
enough to know beforehand that $\ML_{t,s}$ has closed range, which follows from the existence of a parametrix.
\end{proof}

\begin{proposition}
	\label{Iopen}
$I$ is open. 
\end{proposition}
\begin{proof}
The linearization $\ML_{t,s}$ of $N_t$  at $s$ is an isomorphism, and $N_t$ acts smoothly between these same
function spaces, so the implicit function theorem gives the result. 
\end{proof}
\end{subsection}
\end{section}

\begin{section}{Construction of Approximate Solutions}
Our next task is to show that given any triplet $(\ME,\vp,L)$ as in \eqref{complexgeometrymodulispace}, there
exists an admissible Hermitian metric compatible with this data. 

\begin{subsection}{Approximate solutions with Nahm pole boundary data}
We begin with the simpler case where $\mfd(\ME,\vp,L)=\emptyset$, i.e., the holomorphic data lies in the Hitchin section $\MM_{\Hit}$. 
%For a Higgs bundle $(\ME,\vp)$ in the Hitchin component, we will construction an admissible NP Hermitian metric
\begin{proposition}
For every element $(\ME, \vp)$ in the Hitchin component, there exists an $H_0$ such that in unitary gauge relative
to $H_0$, $\Omega_{H_0} = \MO(1)$.
\end{proposition}
\begin{proof}
The Higgs pair $(\ME, \vp)$ has the form \eqref{HitchincomponentHiggsfield}. 
% Let $\ME=K^{-\frac{n}{2}}\oplus K^{-\frac{n-2}{2}}\oplus \cdots\oplus K^{\frac{n}{2}}$ and\begin{equation}
% \begin{split}
% \vp=\begin{pmatrix}
% 0 & \sqrt{B_1}  & 0 &\cdots& 0\\
% 0 & 0 & \sqrt{B_2} &\cdots& 0\\
% \vdots & &\ddots & &\vdots\\
% 0 & & &\ddots  &\sqrt{B_n}\\
% q_{n+1}& q_n & \cdots & q_2 &0
% \end{pmatrix},
% \end{split}
% \end{equation}
Consider $H_0^{(0)}=\exp(-\log y\, \mfe_0)$ and $g=\sqrt{H_0}=\diag(\lam_1,\lam_2,\cdots,\lam_{n+1})$.
By definition of $\mfe_0$, $\lam_i=y^{-\frac{N}{2}+i-1}$, understood as an element of $\End(K^{-\frac{N}{2}+i-1},K^{-\frac{N}{2}+i-1})$. 
Now define $\phi_z=g \vp g^{-1}=(\lam_i \lam_j^{-1}\vp_{ij})$ where $\vp=(\vp_{ij})$, and decompose this endomorphism
as $\phi_z^{\mathrm{mod}}+b$, where 
\begin{equation} 
\begin{split}
\phi_z^{\mathrm{mod}}=y^{-1}
\begin{pmatrix}
0 & \sqrt{B_1} & 0 &\cdots& 0\\
0 & 0 & \sqrt{B_2} &\cdots& 0\\
\vdots &  &\ddots  & &\vdots\\
0 &  & &\ddots  &\sqrt{B_n}\\
0& 0 & \cdots & \cdots &0
\end{pmatrix} \mbox{and}\ \ 
b=\begin{pmatrix}
0 & 0  & 0 &\cdots& 0\\
0 & 0 & 0 &\cdots& 0\\
\vdots &  &\ddots & &\vdots\\
0 &  & &\ddots  & 0\\
y^nq_{n+1}& y^{n-1}q_n & \cdots & yq_2 &0
\end{pmatrix}.
% b=\begin{pmatrix}
% 0 & \sqrt{B_1}  & 0 &\cdots& 0\\
% 0 & 0 & \sqrt{B_2} &\cdots& 0\\
% \vdots &  &\ddots & &\vdots\\
% 0 &  & &\ddots  &\sqrt{B_n}\\
% y^nq_{n+1}& y^{n-1}q_n & \cdots & yq_2 &0
% \end{pmatrix},
\end{split}
\end{equation}
Since $b = \MO(y)$, then in the gauge defined by $g$, 
\[
\Omega_{H_0^{(0)}}=[\phi_z^{\mo},b^{\da}]+[b, (\phi_{z}^{\mo})^\da] =  \MO(1);
\]
even more specifically, the right hand side can be written $F_{H_0^{(0)}} + \MO(y)$.

We may add correction terms to make this error vanish to higher and higher order.  Indeed,
suppose that we have found a Hermitian metric $H_0^{(j)}$ such that $\Omega_{H_0^{(j)}} = F_jy^j + \MO(y^{j+1})$ for
some $j \geq 0$ (so $F_0 = F_{H_0^{(0)}}$ above), 
and define $H_0^{(j+1)} = H_0^{(j)} e^s$.  Using \eqref{quasilinearformequation}, we see that in order 
to show that $\Omega_{H_0^{(j+1)}} = F_{j+1}y^{j+1} + \MO(y^{j+2})$ it suffices to solve the equation 
\[
\gamma(-s)\ML_{H_0^{(j)}} s = -F_j y^j\ \ \mbox{mod terms of order}\ y^{j+1 - \ep},
\]
But $\gamma(-s_j y^j) = \mathrm{Id} + \MO(y)$, and $\ML_{H_0^{(j)}}$ equals 
the normal operator $N(\ML_{H_0})$ to leading order, so it suffices to solve
$N(\ML_{H_0}) s_j y^j = -F_j y^j$ where $s_j$ is just an element of $i\su(E, H_0)$.  
This algebraic equation is solvable at least when $j$ is not an indicial root; in those exceptional cases, 
one must replace $s_j y^j$ by $\tilde{s}_j y^j \log y$ to obtain a solution. 
(The possibility of these extra log factors is the reason we allowed the error $\MO(y^{j+1-\ep})$ above.)
In any case, we can carry out this inductive procedure and then take a Borel sum to obtain a Hermitian endomorphism 
\[
s \sim \sum_{j =0}^\infty s_{j\ell} \, y^j (\log y)^\ell
\]
(with $s_{0\ell} = 0$ for $\ell > 0$ and only finitely many $s_{j \ell}$ nonzero for each $j$), 
such that if we set $H_0 = H_0^{(0)} e^s$, then $\Omega_{H_0} = \MO(y^N)$ for all $N \geq 0$.
\end{proof}

In summary, we obtain
\begin{theorem}
	\label{existadmissble1}
	For any $(\ME,\vp)\in\MM_{\Hit}$, there exists an admissible Hermitian metric.
\end{theorem}
\end{subsection}

\begin{subsection}{Approximate solutions at knot singularities}
Now suppose that we have holomorphic data $(\ME,\vp,L)$ where $S(\ME,\vp,L)=\{p_j,\vfr^j=(\mfr^j_1,\cdots,\mfr^j_n), \ j = 1,
\ldots, N\}$.  As before, we wish to construct an approximate solution near each $p_j$.  As will be clear from this construction,
it involves very little more effort to construct a solution to infinite order as one which solves the equation only up to
first order.   Note that the construction of approximate solutions at points of $F_{\mathrm{or}}$ is in fact completely
local and algebraic (or really, involving a finite jet in the normal direction at each point).  Thus we may 
find a good approximate solution near each knot face $F_j$ and then proceed with the previous construction
to obtain an infinite order solution along the remaining part o the boundary.    It therefore suffices to focus 
on this construction near each $p_j$ separately, and so we drop the index $j$ below.

Fix a small disk $\calU$ centered at $p$ in $\Si$, with local holomorphic coordinate $z$, and work in spherical coordinates 
on $\Si \ti \RP$.  Write $\vp=\vp_z dz$ and set $L_1:=L$, $L_{i+1}:=\vp_z(L_i)$. By assumption, the map 
\begin{equation}
L_1\we L_2\we \cdots \we L_{n+1}\xrightarrow{1\we \vp_z\we \cdots \we \vp_z^n} \det \ME
\end{equation}
fails to be an isomorphism in this neighborhood precisely at $p$.

\begin{lemma}
\label{lemmalocalformKnot}
\label{canonicalformUj}
There exists a local holomorphic frame for $\ME$ in $\calU$ such that 
\begin{equation}
\begin{split}
\vp_z =\begin{pmatrix}
\st & z^{\mfr^j_1} & 0 &\cdots& 0\\
\st & \st & z^{\mfr^j_2} & \cdots& 0\\
\vdots & \vdots &\ddots &  &\vdots\\
\st & \vdots & &\ddots & z^{\mfr^j_n}\\
\st& \st & \cdots & \cdots &\st
\end{pmatrix}, 
\end{split}
\label{canonialknotform}
\end{equation}
where all of the components labelled with a $\st$ are bounded holomorphic functions.  
\end{lemma}
\begin{proof}
We seek a local holomorphic frame $\{\he_1,\he_2,\cdots,\he_{n+1}\}$ such that for each $k$, 
$\vp_z(\he_{k})\subset span\{\he_1,\cdots,\he_{k+1}\}$. 

First choose a nonvanishing section $\he_1$ of $L$, and extend it to a local holomorphic frame $\{\he_1,e_2,\cdots,e_{n+1}\}$
over $\calU$.   Now write $\vp_z(\he_1)=f\he_1+z^{\mfr_1}\sum_{i=2}^{n+1}g_ie_i$, where at least one of the $g_i$ are nonvanishing
at $z=0$.  Setting $\he_2:=\sum_{i=2}^{n+1}g_ke_2$, then we have arranged that $\vp_z(\he_1)=f\he_1+z^{\mfr_1}\he_2$, or
equivalently, $\vp_z(\he_1)=z^{\mfr_1}\he_2 \mod \{\he_1\}$. The integer $\mfr_1$ is the order of vanishing of the map 
$L^{\otimes 2}\xrightarrow{1\we\vp_z}\we^2\ME$.   Next, $\vp_z(\he_2)=f_1\he_1+f_2\he_2+z^{\mfr_{2}}(\sum_{i=3}^{n+1}g_ie_i)$
where at least one of these new coefficient functions $g_i$ do not vanish at $0$, so we define $\he_3=\sum_{i=3}^{n+1}g_ie_i$. 
This process can be continued inductively to obtain a local frame $\{\he_1,\he_2,\cdots,\he_{n+1}\}$, where
$\vp_z (\he_{k})\in z^{\mfr_k}\he_{k+1} + \operatorname{span}\{\he_1,\cdots, \he_{k}\}$, as desired.
\end{proof}

\begin{proposition}
There exists a Hermitian metric $H_{\calU}$ in $\calU$ such that in the corresponding unitary gauge, 
$|\Omega_{H_{\calU}}| = \MO(\psi^{-1}R^{-1})$.
\end{proposition}
\begin{proof}
By virtue of the previous Lemma, write $\vp_z$ as in \eqref{canonialknotform}.  Denote the corresponding model 
solution by $\HM$ and set $g=\sqrt{\HM}$. In unitary frame, $g=\diag(\lam_1,\lam_2,\cdots,\lam_{n+1})$, where 
$|\lam_{k+1}\lam_k^{-1}|\leq C$ by \ref{modelmetricestimate}.  The matrix $\phi_z=g\vp_zg^{-1}$ has components 
$\lam_i\lam_j^{-1}\vp_{ij}$, and decomposes as $\phi_z^{\mo}+b$ where $\phi_z^{\mo}$ is the model 
knot solution and $b = \MO(1)$.  Hence in unitary gauge, $\Omega_{H_0}=[\phi^{\mo}_z,b^{\da}]+[b,\phi_{\bar{z}}^{\mo}]
\sim\MO(\psi^{-1}R^{-1})$.
\end{proof}

To extend this away from the knot, choose a holomorphic frame away from the knot(s) and holomorphic frame near
each $p_j$.  On the overlap near $p_j$ these frames are related by a unimodular gauge transformation 
\begin{equation}
g_j=\exp(\sum_{i=1}^nA_{ij}^{-1}(-\mfr_j\log r)H_i)=z^{\frac{\sum_{i=1}^n(n+1-i)\mfr_i}{n+1}}\diag (1,z^{\mfr_1},z^{\mfr_1+\mfr_2},\cdots,z^{\mfr_1+\cdots+\mfr_n}).
\end{equation}

Using a partition of unity, we then conclude the
\begin{proposition}
For any $(\ME,\vp,L)$, there exists an admissible $H_1$ such that in unitary gauge $|\Omega_{H_1}| =  \MO(\psi^{-1}R^{-1})$
near $F_j$ and $|\Omega_{H_1}| = \MO(y^{-1})$ near $F_{\mathrm{or}}$. 
\end{proposition}

Proceeding even further, we can correct $H_1$ to a Hermitian metric for which $\Omega_H$ vanishes to all orders both
as $y \to 0$ away from the knots and as $R \to 0$.  This involves iteratively solving away the Taylor series in $R$ of
the error term $\Omega_{H_1}$ using the operator $\ML_{H_1}$, which is now an analytic problem on $S^2_+$.  
At each step we solve $\ML (s_j R^j) = \eta_{j-2} R^{j-2}\ \mbox{mod}\ \MO(R^{j-1})$, which is possible unless $j$
is an indicial root; in that case we add a factor of $\log R$ to $s_j R^j$.  A Borel sum of these approximate solutions
yields a Hermitian metric $H_2$ for which $\Omega_{H_2}$ vanishes to all orders as $R \to 0$.  We may finally
carry out the analogous procedure near $y = 0$, and in particular near $\psi = 0$ near the knots. This is now
a pointwise algebraic operation. Taking a further Borel sum leads finally to the Hermitian metric $H$ for which
$\Omega_H$ vanishes to all orders as $R \to 0$ and as $y$ (or $\psi$) tends to $0$. 
\end{subsection}

\begin{theorem}
	\label{existadmissble2}
For any given pair $(\ME,\vp,L)\in\MCK$, there exists an admissible Hermitian metric $H$.
\end{theorem}
% \begin{proof}
% By \eqref{quasilinearformequation}, we have $\Omega_H=\Omega_{H_0}+\gamma(-s)\ML_{H_0}s+Q(s)$, with 
% $|\Omega_{H_0}|\sim \MO(\psi^{-1}R^{-1})$
% \end{proof}

\end{section}

\begin{section}{A Priori Estimates}
To prove closedness of the set $I$, we must show that if $H_j$ is a sequence of solutions corresponding
to a sequence of $t_j \in I$, then there are a priori estimates on the $H_j$ which allow us to take a limit,
which then shows that the limit of the $t_j$ also lies in $I$.   This is done in a sequence of steps where
we bound first the $\calC^0$ then the $\calC^k$ norms of the $H_j$ for every $k$, then establish
uniform decay rates and H\"older estimates near $y=0$ and as $y \to \infty$. 

Before embarking on these nonlinear estimates, we recall a now-standard Lemma about the mapping behavior
of the Laplacian on cylinders.

Define $\MC^{k,\al}_D(\Si \ti \RP)$ to consist of all functions or sections which are uniformly in $\calC^{k,\alpha}$ 
on every strip $\si \ti [t, t+1]$, and which in addition vanish at $y=0$ (hence the subscript `D' for Dirichlet).
We also write $\MC^{k,\al}_{D, -\delta} = e^{-y\delta}\calC^{k,\al}_D$. 

Now fix $\chi \in \calC^{\infty}( \Si \ti \RP)$ with $\chi \geq 0$, $\chi(y) = 1$ for $y \geq 2$ and $\chi(y) = 0$ for
$y \leq 1$.  
\begin{proposition}
\label{scalarlaplacian}
Let $\Delta$ be the scalar Laplacian. Then 
\begin{equation*}
\begin{split}
\Delta: \calC^{k+2,\al}_{D, -\delta}(\Si \ti \RP) \oplus \RR & \longrightarrow \calC^{k,\al}_{-\delta} (\Si \ti \RP) \\
(u,A) & \longmapsto \Delta u + A\Delta (\chi)
\end{split}
\end{equation*}
is an isomorphism.
\end{proposition}
\begin{proof} Results of this type are now quite classical, and indeed this appears explicitly in \cite{MazzeoPollackUhlenbeck} 
for example We provide a brief sketch of the proof.  The operator $\Delta: \calC^{k+2,\alpha}_{D, \tau} \to \calC^{k,\alpha}_{\tau}$ 
is Fredholm so long as $\tau^2$ is not an eigenvalue $\si_j^2$ of $\Delta_\Si$.  The maximum principle shows that this map
is injective when $\tau < 0$, so an argument involving both duality and elliptic regularity shows that it is surjective
when $\tau > 0$ and $\tau$ does not equal one of the values $\si_j$ above.  Since $\si_0 = 0 < \si_1 \leq \ldots$, we
can choose $0 < \delta < \si_1$.  If $f \in \calC^{k,\alpha}_{-\delta}$, there exists a solution $v \in \calC^{k+2,\alpha}_{D, \delta}$
to $\Delta u = f$. The final step is to observe that since $f$ decays at a (small) exponential rate, the solution $v$
has a partial expansion of the form $v = A + \tilde{A} y + \tilde{u}$ with $\tilde{u} \in \calC^{k+2,\alpha}_{D, -\delta}$.
This is a type of elliptic regularity at infinity. The coefficients $A, \tilde{A}$ here are constants. However, the
function $y$ lies in the nullspace of $\Delta$ on $\calC^{k+2,\alpha}_{D, \delta}$ since it vanishes at $y=0$, so
we may as well subtract off this term to see that there exists a solution which is asymptotic to $A + \tilde{u}$
as $y \to \infty$.  The map $f \to A$ is continuous, so there exists a codimension one subspace $H$ of functions $f$
in $\calC^{k,\alpha}_{-\delta}$ for which there exists a unique solution $ u \in \calC^{k+2,\alpha}_{D, -\delta}$ to $\Delta u = f$.
The function $\Delta( \chi)$ does not lie in this subspace, so by choosing the constant $A$ appropriately,
then for any $f \in \calC^{k,\alpha}_{-\delta}$ we can make $f - A \Delta \chi \in H$.
\end{proof}

Recall the notation $\calX^{k,\alpha}_{\mu, \nu, -\delta} = \psi^\mu R^\nu e^{-y\delta}\calC^{k,\alpha}_{\ie}(\Si \ti \RP; \isu(E, H))$, 
%We also define the conormal space $\calA^{\mu, \nu, -\delta} = \cap_{k \geq 0} \calX^{k,\alpha}_{\mu, \nu, -\delta}$.  
If the index $\nu$ is omitted, this connotes the space without knot singularities, and with weight function $y^\mu$. 

\begin{subsection}{$\calC^0$ estimate}
\begin{proposition}  If $s$ is a Hermitian endomorphism which satisfies $N_t(s) = 0$, i.e., equation \eqref{eqmethodofContinuation}, 
then there exist a constant $C$ depending only on $H_0$ such that 
\begin{equation}
|s|_{\calC^0(\Si\ti\RP)}\leq C.
\end{equation}
\label{C0estimate}
\end{proposition} 
\begin{proof}
Taking the inner product of \eqref{quasilinearformequation} with $s$, where $H=H_0e^s$, gives
\begin{equation}
\Delta|s|^2+|\sqrt{\gamma(-s)}\na s|^2+t|s|^2+\langle \Omega_{H_0},s\rangle =0;
\label{ausefulformofmethodofcontinouationequation}
\end{equation}
here $\Delta= -\del_y^2 + \Delta_{\Sigma}$ and $|\sqrt{\gamma(-s)}\na s|^2=\sum_{i=1}^3|\sqrt{\gamma(-s)}\na_i s|^2$,
which then implies 
\[
\Delta |s|^2+t|s|^2\leq -\langle \Omega_{H_0},s\rangle  \Longrightarrow 
\Delta|s|^2\leq M|\Omega_{H_0}|
\]
for some constant $M > 0$. By Proposition \ref{scalarlaplacian}, there exists $u \in \calC^{2,\alpha}_{D, -\delta}$ and 
$A \in \RR$ such that $\Delta (u-A\chi)=|\Omega_{H_0}|$.   Hence
\[
\Delta(|s|^2-Mu+AM\chi - \ep y)\leq 0.
\]
Since $s$ and $u$ decay as $y \to \infty$ and vanish at $y=0$, and $\chi$ is bounded, we see that $|s|^2 -Mu + AM\chi \leq  \ep y$
for any $\ep > 0$, hence $|s|^2 \leq  M (\sup |u| + |A|)$.  Since $u$ and $A$ depend only on $H_0$, this gives the desired bound.
\end{proof}
\end{subsection}

\begin{subsection}{Morrey and Campanato Spaces}
Because of the structure of the equation, the higher derivative estimates are obtained using two slightly less familiar 
scales of function spaces, which we now review.  We also need to introduce certain scale-invariant
modifications of these spaces near a boundary, which are required in estimates near $y=0$. 

We first define the Morrey space $L^{p,\lambda}(\calU)$, where $\calU$ is an open set of $\RR^n$. This is the Banach space 
of functions for which
$$
||f||_{p,\lambda} :=  \sup_{x \in U \atop r > 0}  \left(r^{-\lambda} \int_{B_r(x) \cap \calU} |f|^p\right)^{1/p} < \infty.
$$
There is a well-known embedding theorem, cf. \cite{HumbertoNatashaStefan}:
\begin{proposition} For any function $f$ defined on an open set $\calU \subset \RR^n$ with $\nabla f \in L^{2, n-2+2\alpha}$, the
Morrey norm of $\nabla f$ bounds the $\calC^{0,\alpha}$ seminorm of $f$: 
\[
[f]_{\calC^{0,\alpha}} \leq C ||\nabla f||_{L^{2, n-2+2\alpha}}.
\] 
In particular, when $n=3$, $L^{2, 1+2\alpha} \subset \calC^{0,\alpha}$. 
\end{proposition}

Recall also the closely related Campanato spaces $\ML^{p,\lam}$, defined as the set of all $f \in L^p(\calU)$ for any
open set $\calU$ such that
\begin{equation}
[f]_{\ML^{p,\lam}(\calU)}:=\sup_{x\in U,r>0}(r^{-\lam}\int_{B_r(x)\cap \calU}|f-\bar{f}_{x,r}|^p)^{\frac 1p} < \infty;
\end{equation}
here $\bar{f}_{x,r}$is the average of $f$ on $B_r(x)$.   Campanato spaces will actually not appear explicitly below.
They do arise in a crucial estimation in \cite{Hildebrandt1985} which is used in the interior estimates in Proposition \ref{43} 
and again in the estimates of Section 8.6 below.

As noted earlier, we shall need to use adapted versions of these spaces near $y=0$.   The motivation is
similar to that for the $\ie$-H\"older spaces $\calC^{k,\alpha}_\ie$ from Definition
\ref{defhold0}. Namely, if $u$ is defined in a  ``Whitney cube'' $Q$, e.g. a ball of radius $y_0/2$ centered
at a point $(z_0, y_0)$ (where $y_0$ is small), then $u_\lam(z,y) = u(\lam z, \lam y)$ is defined on $Q_{1/\lam}$.
We define norms which have the property that the sizes of $u|_Q$ and $u_\lam|_{Q_{1/\lam}}$ are the same (or at
least comparable.   Definition~\ref{defhold0} illustrates this for the H\"older norm and defines the scale
of spaces which is usually denoted $\calC^{k,\alpha}_0$ if we consider only scalings near $y=0$ away from knots and
$\calC^{k,\alpha}_{\ie}$ if we also incorporate scalings near the knot singularities. 

In a very similar way we define scale-invariant Morrey and Camanato spaces $L^{p,\lambda}_0$ and $\ML^{p,\lambda}_0$: 
\[
\|u\|_{L_0^{p,\lam}}:=\sum_{B_\mu}[u_\mu]_{L^{p,\lam}},\;and\;\|u\|_{\ML_{0}^{p,\lam}}:=\sum_{B_\mu}\|u_\mu\|_{\ML^{p,\lam}}.
\]
Since 
\begin{equation*}
\|u_{\mu}(x,y)\|_{L^{p,\lam}(B_1)}=\sup_{x\in B_1,r>0} (r^{-\lam}\int_{B_r}|u_\mu|^p)^{\frac{1}{p}} 
=\sup_{x\in B_\mu,r>0}(\mu^{\lam-n}r^{-\lam}\int_{B_r(x)}|u|^p)^{\frac1p}
\end{equation*}
we might equally well define this scale-invariant Morrey norm by
\[
\|u\|_{L^{p,\lam}_0(M)}=\sup_{x\in M,r>0}(y^{\lam-n}r^{-\lam}\int_{B_r(x)}|u|^p)^{\frac{1}{p}},
\]
(where $y$ is the distance to the boundary). 
Similarly, 
\[
\|u\|_{\ML^{p,\lam}_0(M)}=\sup_{x\in M,r>0}(y^{\lam-n}r^{-\lam}\int_{B_r(x)}|u-\bar{u}_{x,r}|^p)^{\frac{1}{p}}.
\]
These scale-invariant Morrey spaces appear explicitly in Section 8.6 while, as noted earlier,
the scale-invariant Campanato spaces are required implicitly in these arguments as we explain there.

Immediately from the interior estimate we obtain  
\begin{proposition}
$[u]_{C^{0,\al}_0} \leq C\|\na u\|_{L_0^{2,n-2+2\al}}$.
\end{proposition}
\end{subsection}	
	
\begin{subsection}{Interior $\calC^k$ estimate}
We now discuss the interior a priori estimates for higher derivatives of a solution $s$ of \eqref{eqmethodofContinuation}.

We now state the sequence of results which lead to the interior estimates, referring their proofs to the cited sources.
\begin{proposition}{\cite{BandoSiu1994,WalpuskiJacob}}
\label{C0alphainteriorestimate}
If $H=H_0e^{s}$ are defined on a ball $B_2 \subset \Si \ti \RP$ and $B_1 \subset B_2$ is a slightly smaller ball, then
\begin{equation}
|s|_{\calC^{0,\al}(B_1)}  \leq C \|\nabla s\|_{L^{2,2\al}(B_1)} C' (\|s\|_{C^0(B_2)}+|\Omega_H|_{C^0(B_2)}).
\end{equation}
\end{proposition}

Next is a well-known estimate due to Hildebrandt. 
\begin{proposition}{\cite{Hildebrandt1985}}
\label{43}
\label{C1betainteriorestimate}
If $\calV \subset \calU$ are two open sets, and $s$ is a solution to an equation of the form
$$
\Delta s=A+B(\nabla s)+C(\nabla s\otimes \nabla s),
$$
where the coefficient functions $A,B,C$ are bounded in $\calC^0$, then for any $\alpha \in (0,1)$ there
exists a $\beta \in (0,\alpha)$ such that
$$
\|\nabla s\|_{\calC^{0,\beta}(\calV)}\leq C\|\nabla s\|_{L^{2, 2\al}(\calU)}
$$
for some constant $C$ which depends only on the volumes of $\calU$ and $\calV$. 
\end{proposition}

Finally, we invoke an estimate of the type used by Bando and Siu \cite{BandoSiu1994} to complete the boot-strapping. 
\begin{proposition}
\label{interiorckbound}
For any $ k \in \mathbb N$ and $T > 0$, there exists a constant $C_{T,k}$ such that 
\begin{equation}
\|s\|_{\calC^k(\Si\ti[ T,+\infty))}\leq C_{T,k} 
\end{equation}
\end{proposition}

Notice that we have now obtained uniform bounds also as $y \to \infty$.  Later we also establish a uniform
decay rate.
\end{subsection}

\begin{subsection}{Uniform decay at $y=0$} 
In this subsection, we show that the sequence of solutions $s_j$  satisfies a uniform decay rate
$|s_j| \leq C y^\alpha$.  This is nonstandard because of the singular boundary condition, and the
idea involves a scaling analysis. 

\begin{lemma}
\label{localMorreybound}
Fix any point $p = (z_0,y_0)$ in $\Si \ti (0,1)$ and let $r_0= \tfrac12 y_0$.   Assuming $r_0$ is small,
we use a local coordinate $x \in \RR^3$ centered at $p$ in the ball $B_{r_0}(p)$, and then define
%Let $B_r(x)$ be the radius $r$ ball centered at $x$, we define the funcational $f_x(r)$ as:
\begin{equation}
f_p(r):=\int_{B_r}G_p(x)|\nabla s(x)|^2,
\end{equation}
where (here and below) $B_r= B_r(p)$ and $G_p(x) = |x|^{-1}$ is ($4\pi$ times) the three-dimensional flat Green function 
and $|\nabla s|^2$ is short for $\sum_{i=1}^3|\nabla_i s|^2$.

Then there exist constants $C>0$ and $\al \in (0,1)$ indepndent of $p$ such that for every $r < r_0/2$, 
\begin{equation}
f_p(r)\leq C(\|s\|_{L^{\infty}(B_{2r})}+\|\Omega_H\|_{L^{\infty}(B_{2r})}+\|\Omega_{H_0}\|_{L^{\infty}(B_{2r})})r^{2\al}.
\end{equation}
\end{lemma}
\begin{proof} We proceed in a series of steps.

\medskip

\noindent {\bf Step 1.}  We first show that $f_p(r)\leq C$, where $C$ depends on $\|s\|_{L^{\infty}(B_{2r})}$ and $\Omega_{H_0}$. 

Let $\chi \in \calC^\infty$ be a smooth nonnegative function which equals $1$ on $[0,r]$, vanishes outside $[0,2r]$, and
with $0 \leq \chi(t) \leq 1$ for all $t$.  The inequality $|\nabla s|^2 \leq C (1-\Delta |s|^2)$, which follows from
\eqref{ausefulformofmethodofcontinouationequation}, gives 
\begin{equation}
\begin{split}
f_p(r)& \leq \int_{B_{2r}(x)}\chi G |\nabla s|^2 \leq C\int_{B_{2r}} \chi G(1-\Delta|s|^2)\\
& \leq C \left(-|s|^2(p)+r^2+\int_{B_{2r}} (\nabla \chi \nabla G|s|^2+\Delta \chi G |s|^2) \right)\\
& \leq C( r^2+r^{-3}\int_{B_{2r} \backslash B_r} |s|^2) \leq C.
\end{split}
\end{equation}

\medskip

\noindent {\bf Step 2.}  There exist constants $\gamma \in (0,1)$ and $ K > 0$ depending on $\|s\|_{L^\infty(B_2)}$
such that $f_p(r)  \leq \gamma f_p(2r)+K r^2$. 

Indeed, 
%Fixed the ball $B_{2r}(x)$ and $B_r(x)$, for $r$ small, we can assume $E$ is locally trivial on $B_{2r}(x)$, let $i\mathfrak{su}(2)$ be the adjoint bundle over $B_{2r}(x)$. 
%Given $s$ over $B_{2r}(x)$, we 
consider the operator 
\begin{equation}
\begin{split}
\MW: \calC^{1,\alpha}_D(B_{2r}; i\mathfrak{su}(2) ) & \longrightarrow  \calC^{0,\alpha}\Omega^1(B_{2r}, i\mathfrak{su}(2))\oplus 
\calC^{0,\alpha}\Omega^0(B_{2r}; i\mathfrak{su}(2))\\ 
\eta & \longmapsto (\MD_1 \eta +[\vp,\eta],\MD_3\eta), 
\end{split}
\end{equation}
where the subscript $D$ indicates Dirichlet boundary conditions.   This is a left-elliptic operator, i.e., its symbol is injective,
so its nullspace $\kappa$ is finite dimensional and consists of sections smooth up to the boundary. Note that 
since $\MD_1 \eta \in \Omega^{0,1}$ and $[\vp, \eta] \in \Omega^{1,0}$, $\eta \in \kappa$ implies that
$\MD_j \eta = 0$, $j = 1, 2, 3$. Denote by $\Pi$ the $L^2$ orthogonal projection of $\calC^{1,\alpha}_D(B_{2r}; \isu(2))$ 
onto $\kappa$ and  $\kappa^\perp$ its orthogonal complement, both determined relative to $H_0$. 

Define $\bar{s}:=\Pi s$ and $e^\si:=e^se^{-\bar{s}}$. The Baker-Campbell-Hausdorff formula implies that 
$\si=s-\bar s+\frac{1}{2}[s,\bar{s}]+\cdots$, where the reminder is a sum of terms, each a combination of Lie brackets 
of $s$ and $\bar{s}$. Since $[s,\bar s]=[s-\bar s,s]$, we obtain that
\[
|\nabla s|^2 \leq C |\nabla \sigma|^2,\ \ |\sigma|^2 \leq C |s-\bar s|^2,
\]
where $C$ depends on $|s|_{L^{\infty}(B_{2r})}$. 

Using $He^{-\bar s}=H_0e^{\sigma}$ in \eqref{Keyequation2} %\eqref{ausefulformofmethodofcontinouationequation} 
gives 
\begin{equation}
\langle \Omega_{H_0e^{\sigma}}-\Omega_{H_0}, \sigma\rangle =\Delta|\sigma|^2+\sum_{i=1}^3|v(s)\nabla_i \sigma|^2.
\end{equation}
Observe also that $\Omega_{H_0e^{\si}}=\Omega_{He^{-\bar{s}}}=Ad(e^{-\bar{s}})\Omega_{H}$, the last equality following 
from \eqref{quasilinearformequation} since $\MW(\bar{s}) = 0$. Since $v(s)=\sqrt{\frac{e^{\ad_s}-1}{\ad_s}}$, 
$\sqrt{\frac{e^x-1}{x}}\geq C\frac{1}{\sqrt{1+|x|}}$ and $\Ad(e^{\frac{s}{2}})\Omega_H+ts=0$, we obtain
\begin{equation}
|\nabla \si|^2  \leq C(|\Omega_H||\si|+|\Omega_{H_0}||\si| -\Delta |\si|^2) \leq C (1-\Delta|\si|^2),
\end{equation}
where the constant depends on $|\Omega_{H_0}|_{L^{\infty}}$ and $|s|_{L^{\infty}}$.
    
Using the same cutoff function $\chi$ as before, then this, together with the Poincar\'e and Kato inequalities, yields
\begin{equation}
\begin{split}
f_p(r) = \int_{B_r(x)}G|\nabla s|^2& \leq C \int_{B_{2r}(x)}\chi G(-\Delta |\si|^2)+ \chi G\\
   & \leq C r^2+ C\int_{B_{2r}(x)\slash B_r(x)}  \Delta(\chi G)|\si|^2\\
   & \leq C  r^2- C |\si|^2(p)+ C r^{-3}\int_{B_{2r}(x)\slash B_r(x)} |\si|^2\\
   & \leq C r^2+C r^{-3}\int_{B_{2r}(x)\slash B_r(x)}|s-\bar s|^2\\
   & \leq C r^2+ C r^{-1}\int_{B_{2r}(x)\slash B_r(x)}|\MW s|^2\;\;(\text{Poincar\'e inequality for $\MW$})\\
   & \leq C r^2+C r^{-1}\int_{B_{2r}(x) \slash B_r(x)}|\nabla s|^s\;\;(\mbox{since}\ |\nabla s|^2=\sum_{i=1}^3|\nabla_i s|^2=2|\MW s|^2)\\
   & \leq C r^2+ C \int_{B_{2r}(x) \slash B_r(x)}G|\nabla s|^2 \leq Cr^2 + C (f_{2p}(r) - f_p(r)).
   \end{split}
   \end{equation}
This gives the desired inequality $f_p(r)  \leq \gamma f_p(2r)+K r^2$ for $\gamma=\frac{C}{C+1}<1.$

\medskip

\noindent {\bf Step 3:} Finally, $f_p(r)\leq Cr^{2\alpha}$ for some $C$ which depends on $||s||_{L^\infty(B_{2r})}$. 

Assume that $1/2 < \gamma < 1$ and set $g(r) = f_p(r) + \frac{K}{4\gamma -1} r^2$. 
Then the `doubling inequality' for $f_p$ from Step 2 implies that 
\[
g(r)\leq \gamma f(2r) + Kr^2 + \frac{K}{4\gamma-1}r^2 = \gamma ( f(2r) + \frac{K}{4\gamma-1} (2r)^2) = \gamma g(2r),
\]
so more generally, $g(r) \leq \gamma^k g(2^k r)$.  This in turn yields the estimate $f_p(r)\leq Cr^{2\alpha}$ for some $\alpha \in (0,1)$.
\end{proof}

\medskip

We finally deduce the uniform decay estimate. 

\begin{proposition} If $N_t(s) = 0$, then there exist constants $C > 0$ and $\alpha\in (0,1)$ such that 
	\label{uniformdecayestimate1}
\begin{equation}
|s|_{\calC^{0}(\Si\ti(0,1])}\leq Cy^{\al}.
\end{equation}
\end{proposition}
\begin{proof}
For any $p\in\Si\ti (0,1]$, let $r = r_p$ be half the $y$-coordinate of $p$. Then Lemma \ref{localMorreybound} gives
the Morrey estimate $\|\nabla s\|_{L^{2,1+2\al}(B_{r})}\leq C$, where $\nabla$ is a connection with $1/y$ singularity. 
By the Kato inequality, 
$$
\|d|s|\|_{L^{2,1+2\al}(B_{r})}\leq C,
$$
where $C$ depends on the $H_0$ norm of $s$, but not on $p$. 

The crucial point is that the H\"older seminorm $[s]_{0,\alpha}$ and $||d|s| ||_{2, 1+2\al}$ scale in precisely the
same way if we dilate $B_r(p)$ by the factor $1/2r_p$ to a ball of radius $r/2r_p < 1/2$ centered at some point $(\bar{p}, 1)$.
On this larger `standard' ball we take advantage of the embedding $L^{2,1+2\al}(B_{1/2})\hookrightarrow \calC^{\alpha}(B_{1})$
and then rescale back to $B_r(p)$ to get $[s]_{C^{\al}(B_r(p))}\leq C$, with constant independent of $r < r_p$. 

Writing $\Si\ti(0,1) \ni p$ in local coordinates as $(z_0,y_0)$, we define a sequence of points $p_{i}=(z_0,(\frac{2}{3})^iy_0)$ 
and define $s_i:=s(p_i)$ and $B_i:=B_{r_i}(p_i)$, where $r_i = r_{p_i} = \frac{1}{2}(\frac{2}{3})^i$. The previous bound implies 
\[
|s_{i}-s_{i+1}|\leq C|p_i-p_{i+1}|^{\al}\leq C(2/3)^{i\al} |y_0-\tfrac13 y_0|^{\al} \leq C' (2/3)^{i\al} y_0^\al,
\]
where $C'$ depends only on $||s||_{L^\infty}$ and $||\Omega_{H_0}||_{L^\infty}$. Therefore, for each $i$, 
$$
|s_0-s_i|\leq \sum_{j=0}^{i-1}|s_{j}-s_{j+1}|\leq C'\sum_{j=0}^{i-1}(2/3)^{j\al}y_0^{\al}\leq C''y_0^{\al}.
$$
As $s$ is continuous and $|s|_{y=0}=0$, we see that $\lim s_i  = 0$, so $|s(p_0)| = |s_0| \leq C y_0^{\al}$. 
This gives the uniform decay rate.
\end{proof}
\end{subsection}

\begin{subsection}{Uniform decay at infinity}
We now show that the sections $s_{t_i}$ decay uniformly as $y \to \infty$. 

First, \eqref{ausefulformofmethodofcontinouationequation} and Proposition \ref{C0estimate} show that
\begin{equation}
\begin{split}
|\na s|^2 \leq C(||\Omega_{H_0}||_{L^\infty}-\Delta |s|^2),
\label{exponentialdecays}
\end{split}
\end{equation}
combining with Proposition \ref{interiorckbound}, we obtain $\int_{\Si\ti(L,+\infty)}|\na s|^2\leq C$ for any $L \geq 1$ and with $C$ independent of $L$. 

Next, the irreducibility of the Higgs pair $(\ME,\varphi)$ on any slice $\Si_y$  implies
\begin{equation}
\int_{\Sigma_y} |s|^2 \leq C \int_{\Sigma_y}|\MD_1 s|^2+|\MD_2 s|^2  \leq C\int_{\Sigma_y}|\na s|^2,
\end{equation}
hence the (integral) decay rate of $s$ is controlled by that of $|\nabla s|$.

\begin{proposition}
Assuming that $||s||_{L^\infty} + ||e^{-\delta y}\Omega_{H_0}||_{\calC^k}\leq C_k$ for any $k \geq 0$,  then for all $k$, 
$||e^{-\delta y}s||_{\calC^k}\leq C_k'$. 
\end{proposition}
\proof
Integrating \eqref{exponentialdecays} over $\Sigma\ti (L,+\infty)$ gives
\[
F(L) := \int_{\Sigma\ti (L,\infty)} |\na s|^2\leq e^{-\delta L}+ \int_{\Sigma\ti \{L\}}|\na s|^2,
\]
or equivalently, $F(L)\leq e^{-\delta L}-F'(L)$. This integrates to $F(L) \leq C e^{-\delta L}$ (presuming $\delta < 1$) and hence 
\begin{equation}
\int_{\Si\ti(L,\infty)}(|s|^2 + |\na s|^2) \leq C e^{-\delta L}.
\end{equation}

Finally, standard interior estimates for the equation $N_t(s) = 0$ on each block $\Si \ti [A, A+1]$ yield that for each $k$, 
\[
\int_{\Si \ti [L, \infty)}  |\nabla^k s| \leq C_k e^{-\delta L},
\]
and the pointwise decay now follows from Sobolev embedding. 
\qed

\end{subsection}

\begin{subsection}{Decay at the Boundary with Higher Regularity}
The equation $N_t(s) = 0$ has the form $\Delta s=A(\Omega_H)+C(\nabla s \otimes \nabla s)$.  We shall first
consider the case where there are no knot singularities.  In this case, to emphasis the Nahm pole boundary condition, we write this equation as
\[
(\Delta+\frac{1}{y^2}) s=A(\Omega_H)+C((\nabla+\frac1y) s \otimes (\nabla +\frac1y)s)
\]
or, equivalently, multiplying through by $y^2$, 
\begin{equation}
y^2\Delta s=y^2A(\Omega_H)+C(y\nabla s \otimes y\nabla s) 
\label{rescleq}
\end{equation}
to emphasize its scaling properties. 

Fix a ball $B_\lam$ of radius $\lam$ and with center at distance $4\lam$ from the boundary, and suppose that $s$ 
solves \eqref{rescleq}.  Restricting \eqref{rescleq} to $B_\lam$, define the rescaled function $s_\lambda(z,y)=s(\lambda z,\lambda y)$.
Setting $K =y^2A(\Omega_H)$, then $s_\lam(z,y)$ satisfies 
\begin{equation}
y^2\Delta s_\lam=K_\lam+C(y\nabla s_\lam \otimes y\nabla s_\lam)
\end{equation}
on $B_1$.  We can also defined rescalings of the Hermitian metrics $H$, $H_0$. Applying Lemma \ref{localMorreybound} to 
$s_\lam$, $H_\lam$ and $(H_0)_\lam$ and the local Morrey estimate on $B_2$ yields
\begin{lemma}
\label{rescaleMorreylocalestimate}
\[
|s_\lam|_{\MC^\al(B_1)} \leq C (\|s\|_{L^\infty(B_2)}+\|\Omega_{(H_\lam)}\|_{L^\infty(B_2)}+\|\Omega_{(H_0)_\lam}\|_{L^\infty(B_2)}),
\]
where all norms are with respect to $(H_0)_\lam$.
\end{lemma}
\begin{proof}
Let $\na^{\lam}$ be the Chern connection for $(H_0)_\lam$. Then by Lemma \ref{localMorreybound}, 
\begin{equation}
\|\na^\lam s_\lam\|_{L^{2,1+2\al}(B_1)}  \leq C( \|s\|_{L^\infty(B_2)}+\|\Omega_{(H_\lam)}\|_{L^\infty(B_2)}+\|\Omega_{(H_0)_\lam}\|_{L^\infty(B_2)})
\end{equation}
Applying the Kato inequality $\|d|s_\lam|\|_{L^{2,1+2\al}}\leq \|\na^\lam s_\lam\|_{L^{2,1+2\al}}$ and Morrey embedding,
and noting that $\na^\lam$ is smoothly convergent as $\lam \to 0$, we obtain the assertion.
\end{proof}

For Proposition \ref{uniformdecayestimate1}, we have already established that $|s| \leq Cy^{\ep}$, so $|s_\lam|_{C^0(B_1)} \leq C \lam^\ep$, and in addition, 
$\Omega_{H_\lam}=\lam^2(\Omega_H)_\lam,$ $\Omega_{(H_0)_\lam}=\lam^2(\Omega_{H_0})_\lam$.  We can now deduce
the following result from the local interior estimates:
\begin{proposition} $[s]_{y^\ep\MC_0^{1,\al}} \leq C$ where $C$ is independent of $s$. 
\end{proposition}
\begin{proof}
Applying Lemma \ref{rescaleMorreylocalestimate} on $B_2$ shows that 
\begin{equation}
[s_\lam]_{\MC^\al(B_1)} \leq C \lam^2||(\Omega_{H_0})_\lam||_{\MC^0(B_2)}+ C\lam^2||(\Omega_{H})_\lam||_{\MC^0(B_2)}+||s_\lam||_{\MC^0(B_2)}.
\label{previousarguement}
\end{equation}
We have arranged that $\Omega_{H_0}$ vanishes at infinite order in $y$. % so $||(\Omega_{H_0})_\lam|| \leq C \lam^2$. 
Next, $\Omega_{H}+ts=0$, so using $||s||_{\MC^0} \leq C y^{\ep}$ we see that $||(\Omega_{H})_\lam||_{\MC^0} \leq C \lam^{\ep}$,
and hence $[s_\lam]_{\MC^{0,\al}} \leq C \lam^{\ep}$, or in other words, $[s]_{y^{\ep}\MC_0^{0,\al}}\leq C$.

Now apply Proposition \ref{C1betainteriorestimate}, Hildebrandt's $\MC^{1,\beta}$ estimate, to get 
\[
[s_\lam]_{\MC^{1,\beta}}\leq C\|y\na s_\lam\|_{L^{2,1+2\al}}\leq C y^{\ep} \Longrightarrow [s]_{y^\ep\MC_0^{1,\al}}\leq C.
\]
\end{proof}

We may finally apply standard bootstrapping to obtain the following
\begin{theorem}
Suppose that $N_t(s) = 0$ and $H_0$ has a Nahm pole but no knot singularitie. Let $\kappa$ be the first positive 
indicial root of $\ML_{H_0}$.  Then for all $k \in \mathbb N$ and $\al\in(0,1)$, there is an a priori estimate 
\[
[s]_{y^{\kappa}\MC_0^{k,\al}}\leq C,
\]
where $C$ depends on all the data, but not on $s$. 
%where the constant depends on $k,l,\al$ and $\Omega_{H_0}$, which is the full curvature of the admissible NP Hermitian metric $H_0$.
\end{theorem}

So far we have only obtained this estimate in the absence of knot singularities. However, a very similar sort of rescaling 
holds when there are knots.  In this case we use spherical coordinates near each knot to rewrite \eqref{rescleq} as
\[
R^2\Delta s=R^2A(\Omega_H)+C(R \nabla s \otimes R \nabla s).
\]
Now restrict $s$ to a small `cube' where $\lam \leq R \leq 4\lam$ and $(\theta, \psi)$ lies in some fixed open set  $Q'$ 
in the interior of $S^2_+$.  The dilate $s_\lam( R, \theta, \psi)$ is supported in $\{1 \leq R \leq 4\} \times Q'$, where
the equation is uniformly elliptic, and we can apply the interior estimates exactly as before.    In fact, incorporating 
the previous estimate in the case with no knots, we may in fact let $Q'$ be the entire $S^2_+$.   All of this leads to the 
final result: 
\begin{theorem}
Suppose that $N_t(s) = 0$ and $H_0$ has a generealized Nahm pole with knot singularitie. Let $\kappa$ lie between $0$
ad the first positive indicial root of $\ML_{H_0}$ at each of the knots.  Then for all $k \in \mathbb N$ and 
$\al\in(0,1)$, there is an a priori estimate 
\[
[s]_{R^{\kappa}\MC_{\ie} ^{k,\al}}\leq C.
\]
Combining this with the previous estimate shows that 
\[
||s||_{\calX^{k,\alpha}_{\mu, \nu, -\delta}} \leq C
\]
for some $\mu, \nu, \delta > 0$.
\end{theorem}

\end{subsection}

\begin{comment}
Recalling from \ref{Calpha}, as $\|s\|_{C^0}\leq y^{\alpha}C$, we obtain $\|\nabla s\|_{L^{2,n-2+2\alpha}}\leq y^{\alpha}C.$ Thus $$\|y\nabla s\|_{L^{2,2n-2+2\alpha}}\leq y^{1+\alpha}C.$$

As we have the rescale relation
\begin{equation}
\|(y\nabla s)_\lam \|_{L^{2,n-2+2\alpha}}(B_1)=
\lam^{\alpha-1}\|y\nabla s\|_{L^{2,n-2+2\alpha}}(B_\lam),
\end{equation}
we obtain 
\begin{equation}
\|y\nabla s\|_{L_0^{2,n-2+2\alpha}}\leq y^{\alpha-1}\|y\nabla s\|_{L^{2,n-2+2\alpha}}.
\end{equation}
Thus, we obtain $\|y\nabla s\|_{C^{0,\beta}_0}\leq Cy^{2\alpha}.$
\end{comment}

\begin{subsection}{Higher regularity and Existence}
We conclude this analysis by invoking the regularity theory from \cite{MazzeoWitten2013, MazzeoWitten2017}. 
\begin{proposition} If $s$ is a solution to the \EBE\rm{(or method of continuation equation \eqref{eqmethodofContinuation})} with a Nahm pole, or generalized Nahm pole, singularity
at $y=0$, then $s$ is polyhomogeneous on $(\Si \ti \RP)_{\mathfrak p}$; in other words, it admits a full asymptotic expansion 
\[
s \sim  \sum  s^{(1)}_{j \ell}(z) y^j (\log y)^\ell, \qquad s \sim \sum s^{(2)}_{j\ell} (\theta, \psi) R^{\si_j} (\log r)^\ell 
\]
at each of the boundary faces of $(\Si \ti \RP)_{\mathfrak p}$, with a product type expansion
\[
s \sim \sum s^{(3)}_{j \ell i m}(\theta)  \psi^j (\log \psi)^\ell R^{\si_i} (\log R)^m
\]
at the corners, with all coefficients smooth.
\end{proposition}

The importance of this regularity statement is that the a prior estimates above show that if $t_j$ is a sequence
of points in the set $I$ in \eqref{setI}, with corresponding solutions $s_j$, then there is a subsequence (which we relabel as $s_j$ again)
which is uniformly bounded in $\calX^{k,\alpha}_{\mu, \nu, -\delta}$ for every $k$, and for some $\mu, \nu > 1$, and 
hence convergent in a slightly weaker space. The limit $s$ solves the equation $N_t(s) = 0$ for $t = \lim t_j$. We conclude
\begin{proposition}
\label{Iclosed}
The set $I$ in \eqref{setI} is closed.
\end{proposition}

The regularity theory shows that in fact $s$ is `fully' smooth, which shows, finally, that it gives a suitable
background metric to apply the openness theory from Proposition \ref{Inonempty}, \ref{Iopen}. (If we did not have this higher regularity
statement, it would be necessary to extend the mapping properties to operators $\ML$ with less regular
coefficients, which is of course not a hard task.) We conclude the following existence theorem:
\begin{theorem}
If there exists an admissible Hermitian metric $H_0$ with Nahm pole (with knot) boundary conditions, then there exists a solution 
$H$ to the \EBE, cf. $\Omega_H=0$.
\end{theorem}

Recalling Theorems \ref{existadmissble1} and \ref{existadmissble2}, we conclude
\begin{corollary}
The maps $I_{\mathrm{NP}}:\MM_{\mathrm{NP}}\to\MM_{\mathrm{Hit}}$ and $I_{\mathrm{GNP}}:\MM_{\mathrm{GNP}}\to
\MM^{\mathbb{C}}_{\mathrm{Knot}}$ are surjective.
\end{corollary}

We have now completed the proof of the existence of a solution to the equation $N_0(s) = 0$, or more simply, $\Omega_H = 0$,
corresponding to the prescribed holomorphic data.
\end{subsection}
\end{section}

\begin{section}{Uniqueness}
We prove uniqueness of solutions using convexity of the Donaldson functional.    For any two Hermitian metrics $K$ and $H=Ke^{s}$,
with $\Tr(s)=0$, write 
\begin{equation}
\Omega_{H,K}:=\frac{i}{2} \Lambda \left([\MD_1, \MD_1^{\da}]+[\MD_2,\MD_2^{\da}]\right)  + [\MD_3,\MD_3^{\da}],
\end{equation} 
where $\MD_i^{\da}$ is the conjugate with respect to $H$ defined in Section 2, the substript $K$ is to emphasis that when fixing $K$, we are considering $\Omega=0$ as an equation for $s$. 

We define a Donaldson functional for the \EBE in analogy with the well-known Donaldson functional for the
Hermitian-Yang-Mills equations in \cite{donaldson1985anti, Donaldson1987Infinite, Simpson1988Construction}:
\begin{equation}
\begin{split}
&\MM(H,K)=\int_0^1\int_{\Sigma\ti\RP}\langle s,\Omega(Ke^{us},K)\rangle _K\ \omega\we dy \we du\, ;
\end{split}
\end{equation}
here $\omega$ is the volume form of $\Si$.  This functional reveals the variational structure for the \EBE. Indeed, 
writing $H_t=Ke^{ts}$, then
\begin{equation}
\begin{split}
&\frac{d}{dt}\MM(H_t,K)=\int_{\Sigma\ti\RP}\Tr(\Omega_{H_t,K}s)\omega\we dy, \\
& \frac{d^2}{dt^2}\MM(H_t,K)=\sum_{i=1}^3\int_{\Si\ti\RP}|\MD_is|^2+\int_{\Si\ti\RP} \bar{\partial}\Tr(D_1^{\da}s\we s)+\int_{\Si\ti\RP} \partial_y\Tr(D_3^{\da}s\we s).
\label{Donaldsonfunctionalvariation}
\end{split}
\end{equation}

We now use this to prove injectivity of the maps $I_{\mathrm{NP}}$ and $I_{\mathrm{GNP}}$ from Propositions \ref{INP} and \ref{IGNP}.
%  we define the maps $I_{\mathrm{NP}}:\MM_{\mathrm{NP}}\to\MM_{\mathrm{Hit}}$ and $I_{\mathrm{GNP}}:\MM_{\mathrm{GNP}}\to
% \MM^{\mathbb{C}}_{\mathrm{Knot}}$, where $\MM_{\mathrm{NP}}(\MM_{\mathrm{GNP}})$ is the moduli space of Nahm pole(with knot) solutions to the \EBE, and $\MM_{\mathrm{Hit}}(\MM^{\mathbb{C}}_{\mathrm{Knot}})$ is the corresponding complex data.

\begin{proposition}
Given any element in $\MM_{\mathrm{Hit}}$($\MM^{\mathbb{C}}_{\mathrm{Knot}}$), suppose $H,K$ are two solutions to the \EBE 
with the same singularity type and corresponding to this same set of holomorphic data. Then $H=K$.
\end{proposition}
\begin{proof}   Write $H = K e^s$ and $H_t=Ke^{ts}$.  By the indicial root computations for $\ML$, both near $F_{\mathrm{or}}$ 
and each $F_j$, the order of vanishing of $s$ is greater than $1$, hence the boundary terms in \eqref{Donaldsonfunctionalvariation} 
vanish.  Furthermore, the Higgs pair associated to $(\MD_1,\MD_2)$ is stable, so $\Ker\;\MD_1\cap \Ker\;\MD_2=\emptyset$. 
Hence if we set $m(t):=\MM(H_t,K)$, then $m'(0)=0$ and $m'' > 0$ if $s\not\equiv 0$. However, since $m(0) = m(1) = 0$,
we see that $m \equiv 0$, so $H \equiv K$ after all. 
\end{proof}

\begin{corollary}
The maps $I_{\mathrm{NP}}:\MM_{\mathrm{NP}}\to\MM_{\mathrm{Hit}}$ and $I_{\mathrm{GNP}}:\MM_{\mathrm{GNP}}\to
\MM^{\mathbb{C}}_{\mathrm{Knot}}$ are injective.
\end{corollary}

We have now fully proved our main result, that these maps are bijective.
\end{section}

\medskip

\bibliographystyle{alpha}
\bibliography{references}

\begin{thebibliography}{GPPV17}

\bibitem[AM17]{ManolescuAbouzaid}
Mohammed Abouzaid and Ciprian Manolescu.
\newblock A sheaf-theoretic model for {${SL}(2;\mathbb{C})$} {F}loer homology.
\newblock {\em arXiv preprint arXiv:1708.00289}, 2017.

\bibitem[Ati78]{atiyah1978geometry}
Michael Atiyah.
\newblock Geometry of {Y}ang-{M}ills fields.
\newblock In {\em Mathematical problems in theoretical physics ({P}roc.
  {I}nternat. {C}onf., {U}niv. {R}ome, {R}ome, 1977)}, volume~80 of {\em
  Lecture Notes in Phys.}, pages 216--221. Springer, Berlin-New York, 1978.

\bibitem[Bog76]{Bogomol1976stability}
E.B. Bogomol'nyi.
\newblock The stability of classical solutions.
\newblock {\em Sov. J. Nucl. Phys.(Engl. Transl.);(United States)}, 24(4),
  1976.

\bibitem[BS94]{BandoSiu1994}
Shigetoshi Bando and Yum-Tong Siu.
\newblock Stable sheaves and {E}instein-{H}ermitian metrics.
\newblock {\em Geometry and analysis on complex manifolds}, pages 39--50, 1994.

\bibitem[Cli18]{Taubescompactness}
Taubes Clifford.
\newblock Sequences of {N}ahm pole solutions to the {S}{U}(2)
  {K}apustin-{W}itten equations.
\newblock {\em arXiv preprint arXiv:1805.02773}, 2018.

\bibitem[DF17]{FukayaDaemi2017atiyah}
Aliakbar Daemi and Kenji Fukaya.
\newblock {A}tiyah-{F}loer conjecture: a formulation, a strategy to prove and
  generalizations.
\newblock {\em arXiv preprint arXiv:1707.03924}, 2017.

\bibitem[Don85]{donaldson1985anti}
Simon~K. Donaldson.
\newblock Anti-self-dual {Y}ang-{M}ills connections over complex algebraic
  surfaces and stable vector bundles.
\newblock {\em Proc. London Math. Soc. (3)}, 50(1):1--26, 1985.

\bibitem[Don87]{Donaldson1987Infinite}
Simon~K. Donaldson.
\newblock Infinite determinants, stable bundles and curvature.
\newblock {\em Duke Math. J.}, 54(1):231--247, 1987.

\bibitem[Don92]{Donaldsonboundary}
Simon~K. Donaldson.
\newblock Boundary value problems for {Y}ang-{M}ills fields.
\newblock {\em J. Geom. Phys.}, 8(1-4):89--122, 1992.

\bibitem[GPPV17]{Gukov2017bps}
Sergei Gukov, Du~Pei, Pavel Putrov, and Cumrun Vafa.
\newblock {B}{P}{S} spectra and 3-manifold invariants.
\newblock {\em arXiv preprint arXiv:1701.06567}, 2017.

\bibitem[Guo96]{Guo1996}
Guang-Yuan Guo.
\newblock Yang-{M}ills fields on cylindrical manifolds and holomorphic bundles.
  {I}, {II}.
\newblock {\em Comm. Math. Phys.}, 179(3):737--775, 777--788, 1996.

\bibitem[GW12]{gaiotto2012knot}
Davide Gaiotto and Edward Witten.
\newblock Knot invariants from four-dimensional gauge theory.
\newblock {\em Advances in Theoretical and Mathematical Physics},
  16(3):935--1086, 2012.

\bibitem[He17]{He2017}
Siqi He.
\newblock A gluing theorem for the {K}apustin-{W}itten equations with a {N}ahm
  pole.
\newblock {\em arXiv preprint arXiv:1707.06182}, 2017.

\bibitem[Hil85]{Hildebrandt1985}
Stefan Hildebrandt.
\newblock Harmonic mappings of {R}iemannian manifolds.
\newblock {\em Lecture Notes in Math.}, 1161:1--117, 1985.

\bibitem[Hit87a]{hitchin1987stable}
Nigel Hitchin.
\newblock Stable bundles and integrable systems.
\newblock {\em Duke mathematical journal}, 54(1):91--114, 1987.

\bibitem[Hit87b]{Hitchin1987Selfdual}
Nigel~J. Hitchin.
\newblock The self-duality equations on a {R}iemann surface.
\newblock {\em Proc. London Math. Soc. (3)}, 55(1):59--126, 1987.

\bibitem[Hit92]{hitchin1992lie}
Nigel Hitchin.
\newblock Lie groups and {T}eichm\"uller space.
\newblock {\em Topology}, 31(3):449--473, 1992.

\bibitem[HM17]{HeMazzeo2017}
Siqi He and Rafe Mazzeo.
\newblock The extended {B}ogomolny equations and generalized {N}ahm pole
  boundary condition.
\newblock {\em arXiv preprint arXiv:1710.10645}, 2017.

\bibitem[JW16]{WalpuskiJacob}
Adam Jacob and Thomas Walpuski.
\newblock {H}ermitian {Y}ang-{M}ills metrics on reflexive sheaves over
  asymptotically cylindrical {K}ahler manifold.
\newblock {\em arXiv preprint arXiv:1603.07702}, 2016.

\bibitem[KW07]{KapustinWitten2006}
Anton Kapustin and Edward Witten.
\newblock Electric-magnetic duality and the geometric {L}anglands program.
\newblock {\em Commun. Number Theory Phys.}, 1(1):1--236, 2007.

\bibitem[Man06]{ManolescuKnot}
Ciprian Manolescu.
\newblock Nilpotent slices, {H}ilbert schemes, and the {J}ones polynomial.
\newblock {\em Duke Math. J.}, 132(2):311--369, 2006.

\bibitem[Maz91]{Mazzeo1991}
Rafe Mazzeo.
\newblock Elliptic theory of differential edge operators. {I}.
\newblock {\em Comm. Partial Differential Equations}, 16(10):1615--1664, 1991.

\bibitem[Mik12]{Mikhaylov2012solutions}
Victor Mikhaylov.
\newblock On the solutions of generalized {B}ogomolny equations.
\newblock {\em Journal of High Energy Physics}, 2012(5):112, 2012.

\bibitem[MPU96]{MazzeoPollackUhlenbeck}
Rafe Mazzeo, Daniel Pollack, and Karen Uhlenbeck.
\newblock Moduli spaces of singular {Y}amabe metrics.
\newblock {\em J. Amer. Math. Soc.}, 9(2):303--344, 1996.

\bibitem[MW13]{MazzeoWitten2013}
Rafe Mazzeo and Edward Witten.
\newblock The {N}ahm pole boundary condition.
\newblock {\em The influence of Solomon Lefschetz in geometry and topology.
  Contemporary Mathematics}, 621:171--226, 2013.

\bibitem[MW17]{MazzeoWitten2017}
Rafe Mazzeo and Edward Witten.
\newblock The {KW} equations and the {N}ahm pole boundary condition with knot.
\newblock {\em arXiv preprint arXiv:1712.00835}, 2017.

\bibitem[Nah80]{nahm1980simple}
Werner Nahm.
\newblock A simple formalism for the {B}{P}{S} monopole.
\newblock {\em Phys. Lett., B}, 90(4):413--414, 1980.

\bibitem[Owe01]{Owens2001}
Brendan Owens.
\newblock Instantons on cylindrical manifolds and stable bundles.
\newblock {\em Geom. Topol.}, 5:761--797, 2001.

\bibitem[RSS13]{HumbertoNatashaStefan}
Humberto Rafeiro, Natasha Samko, and Stefan Samko.
\newblock Morrey-{C}ampanato spaces: an overview.
\newblock In {\em Operator theory, pseudo-differential equations, and
  mathematical physics}, volume 228 of {\em Oper. Theory Adv. Appl.}, pages
  293--323. Birkh\"auser/Springer Basel AG, Basel, 2013.

\bibitem[SE15]{Earp2015}
Henrique~N. S\'a~Earp.
\newblock {$G_2$}-instantons over asymptotically cylindrical manifolds.
\newblock {\em Geom. Topol.}, 19(1):61--111, 2015.

\bibitem[Ser01]{Serre2001Complex}
Jean-Pierre Serre.
\newblock {\em Complex semisimple {L}ie algebras}.
\newblock Springer Monographs in Mathematics. Springer-Verlag, Berlin, 2001.
\newblock Translated from the French by G. A. Jones, Reprint of the 1987
  edition.

\bibitem[Sim88]{Simpson1988Construction}
Carlos~T. Simpson.
\newblock Constructing variations of {H}odge structure using {Y}ang-{M}ills
  theory and applications to uniformization.
\newblock {\em J. Amer. Math. Soc.}, 1(4):867--918, 1988.

\bibitem[SS06]{SeidelSmithLinkinvariant}
Paul Seidel and Ivan Smith.
\newblock A link invariant from the symplectic geometry of nilpotent slices.
\newblock {\em Duke Math. J.}, 134(3):453--514, 2006.

\bibitem[TL18]{RyosukeEnergy}
Ryosuke Takahashi and Naichung Leung.
\newblock Energy bound for {K}apustin-{W}itten solutions on
  {$S^3\times\mathbb{R}^+$}.
\newblock {\em arXiv preprint arXiv:1801.04412}, 2018.

\bibitem[UY86]{uhlenbeck1986existence}
Karen Uhlenbeck and S.-T. Yau.
\newblock On the existence of {H}ermitian-{Y}ang-{M}ills connections in stable
  vector bundles.
\newblock {\em Comm. Pure Appl. Math.}, 39(S, suppl.):S257--S293, 1986.
\newblock Frontiers of the mathematical sciences: 1985 (New York, 1985).

\bibitem[Wit12]{witten2011fivebranes}
Edward Witten.
\newblock Fivebranes and knots.
\newblock {\em Quantum Topol.}, 3(1):1--137, 2012.

\bibitem[Wit14]{Witten2014LecturesJonesPolynomial}
Edward Witten.
\newblock Two lectures on the {J}ones polynomial and {K}hovanov homology.
\newblock {\em arXiv preprint arXiv:1401.6996}, 2014.

\bibitem[Wit16]{Witten2016LecturesGaugeTheory}
Edward Witten.
\newblock Two lectures on {G}auge theory and {K}hovanov homology.
\newblock {\em arXiv preprint arXiv:1603.03854}, 2016.

\bibitem[WZ17]{Wei2017Inverse}
Yangjiang Wei and Yi~Ming Zou.
\newblock Inverses of {C}artan matrices of {L}ie algebras and {L}ie
  superalgebras.
\newblock {\em Linear Algebra Appl.}, 521:283--298, 2017.

\end{thebibliography}
\end{document}